\documentclass[12pt]{amsart}
\usepackage[usenames]{color} 
\usepackage{times}
\usepackage{amsfonts}
\usepackage{amsthm}
\usepackage{amsmath}
\usepackage{psfrag}
\usepackage{times}
\usepackage{latexsym,color}
\usepackage{amssymb}
\usepackage{graphicx}

\usepackage{epsfig}
\usepackage{gastex}

\advance \topmargin by -\headheight
\advance \topmargin by -\headsep
\evensidemargin \oddsidemargin
\usepackage{amsfonts, amssymb, amsmath, amsthm}
\newcommand{\R}{\mathbb{R}}
\newcommand{\N}{\mathbb{N}}
\newcommand{\Q}{\mathbb{Q}}
\newcommand{\Z}{\mathbb{Z}}

\newcommand{\T}{\mathbb{T}}

\newcommand{\Ob}{\Omega_{\psi_{\beta}}}

\newcommand{\pb}{\psi_{\beta}}
\newcommand{\tpb}{\tilde{\psi}_{\beta}}
\newcommand{\Pb}{\Psi_{\beta}}

\newcommand{\larr}{\left( \begin{array}{c}}
\newcommand{\rarr}{\end{array} \right) }

\newcommand{\lsqarr}{\left[ \begin{array}{c}}
\newcommand{\rsqarr}{\end{array} \right]}

\newcommand{\inv}{\varprojlim}

\newtheorem{theorem}{Theorem}
\newtheorem{Theorem}[theorem]{Theorem}
\newtheorem{cor}[theorem]{Corollary}
\newtheorem{corollary}[theorem]{Corollary}
\newtheorem{conjecture}[theorem]{Conjecture}
\newtheorem{lemma}[theorem]{Lemma}
\newtheorem{prop}[theorem]{Proposition}
\newtheorem{example}[theorem]{Example}
\newtheorem{remark}[theorem]{Remark}
\numberwithin{equation}{section}

\begin{document}

\title[The Pisot Conjecture]{The Pisot Conjecture for $\beta$-substitutions}

\author[M. Barge]{Marcy Barge}
\address{Department of Mathematics\\
Montana State University\\
Bozeman, MT 59717-0240, USA.}
\email{barge@math.montana.edu}

\date{May15, 2015}

\maketitle

\begin{abstract} We prove the Pisot Conjecture for $\beta$-substitutions: If $\beta$ is a Pisot number then the tiling dynamical system $(\Ob,\R)$ associated with the $\beta$-substitution has pure discrete spectrum. As corollaries: (1) Aritmetical coding of the hyperbolic solenoidal automorphism associated with the companion matrix of the minimal polynomial of any Pisot number is a.e. one-to-one; and (2) All Pisot numbers are weakly finitary.
\end{abstract}

\section{ Introduction.}

A {\em substitution} is a map that takes letters of some finite alphabet to finite words in that alphabet. The space of all bi-infinite words that can arise from infinitely repeating a particular
substitution, equiped with the product topology and the shift operation, is called a {\em substitutive system}. The group $\R$ acts on the suspension (there is a natural choice of roof function) of a substitutive system, resulting in the {\em tiling dynamical system} associated with the substitution. One would like to know how close the tiling dynamical system is to being simply an action of $\R$ by translation on a compact abelian group. The {\em Pisot Substitution Conjectures} assert that, under various hypotheses, the tiling dynamical system should be an almost everywhere one-to-one extension of an action of $\R$ by translation on a finite dimensional torus or solenoid.

As long as the substitution is primitive (under some power of the substitution, every letter appears in the image of a single letter), long words are, on average, stretched under substitution by a constant multiplier called the {\em inflation} of the substitution. The one key hypothesis for the Pisot Substitution Conjectures is that the inflation be a Pisot-Vijayaraghavan (or, simply, Pisot) number - a real algebraic integer greater than 1, all of whose (other) algebraic conjugates lie strictly inside the unit circle. For the tiling dynamical system to have a nontrivial group rotation as a factor it is necessary that the inflation be Pisot (\cite{S2}). The tiling dynamical system of any primitive substitution with Pisot inflation (let's call such a substitution a {\em Pisot substitution}) is always a finite extension of a translation action on a torus or solenoid (\cite{BKw,BBK}) but maybe not an almost everywhere one-to-one extension of such an action. For example, the tiling dynamical system associated with the Thue-Morse substitution, $a\to ab$, $b\to ba$, with inflation $\beta=2$, is, at best, an a.e. two-to-one extension of translation on the 2-adic solenoid.

For tiling dynamical systems, equivalent to being an almost everywhere one-to-one extension of an action by rotation on a torus or solenoid is the property of having {\em pure discrete spectrum} (more on this below). There are two main versions of the Pisot Substitution Conjecture: 
\\
\\
(1) If $\beta$ is Pisot, the tiling dynamical system associated with the $\beta$-substitution has pure discrete spectrum.
\\
\\
(2) The tiling dynamical system of an {\em irreducible}\footnote{ A substitution is irreducible if the characteristic polynomial of its abelianization (or substitution matrix - see Section \ref{connections}) is irreducible over the rationals.} Pisot substitution has pure discrete spectrum.
\\
\\
Given a real number $\beta>1$, the {\em $\beta$-transformation}, $T_{\beta}:[0,1]\to [0,1]$, is given by $T_{\beta}(x)=\beta x-\lfloor \beta x\rfloor$, where $\lfloor\,\,\,\rfloor$ denotes the greatest integer function. If $\beta$ is Pisot, the orbit of 1 under $T_{\beta}$ is finite and so breaks $[0,1]$ into finitely many subintervals. These subintervals map over each other under $T_{\beta}$ in a pattern described by the {\em $\beta$-substitution}, $\pb$ (see Section \ref{background}); it is to this substitution that version (1) above refers. 

The two Pisot Substitution Conjectures are independent: most irreducible Pisot substitutions don't  come from the $\beta$-transformation, and most $\beta$-substitutions are not irreducible. We prove version (1) of the Pisot Substitution Conjecture in the present article.

 In \cite{Ra} G. Rauzy constructed 2-dimensional fractal tiles, called {\em Rauzy tiles}, that tile the plane both periodically and non-periodically. His construction is based on the `tribonacci substitution' - the $\beta$-substitution associated with the largest root of $x^3-x^2-x-1$. Using arithmetical properties of $\beta$-numeration, Thurston (\cite{T}) described a general method for obtaining a self-similar `shingling' of the plane based on any degree three Pisot unit $\beta$. His shingling is a multi-tiling, `Galois dual' to the tilings of $\R$ provided directly by the $\beta$-substitution, in which almost every point in the plane lies in exactly $r$ tiles for some $r\ge1$. The degree, $r$, of the shingling is 1 if and only if the shingling is a tiling, and, it turns out, if and only if the tiling dynamical system associated with the $\beta$-substitution has pure discrete spectrum (\cite{A2}). Thurston mentions that when these shinglings are, in fact, tilings, they are associated with Markov partitions for hyperbolic toral automorphisms. This procedure for constructing Markov partitions was subsequently developed by Bertrand-Mathis (\cite{B-M}) and Pragastis (\cite{Pr}) and underlies the `arithmetical coding' of Kenyon, Vershik, Sidorov, Schmidt, and others (\cite{KV, Si1,Si2,Sch2,Ver1,Ver2,Bor} and see Section \ref{coding and W} below). 
 
 Solomyak showed in \cite{S1} that a purely arithmetical condition  - the so-called Property (F),  that every positve element of the ring $\Z[1/\beta]$ have finite $\beta$-expansion - is enough to guarantee pure discrete spectrum of the tiling dynamical system associated with the $\beta$-substitution and, with Frougny (\cite{FS}), described a collection of polynomials with dominant root a Pisot number satifying Property (F). Akiyama and Sadahiro (\cite{AS}) and Akiyama (\cite{A1}) showed that Thurston's shingling produces a tiling when the Pisot unit $\beta$ satisfies Property (F). In \cite{H} Hollander introduced the weaker Property (W) (see Section \ref{coding and W} below), which Akiyama (\cite{A2}) subsequently proved to be equivalent to pure discrete spectrum for the 1-d tiling system (and to the Galois dual shingling being a tiling). In \cite{Si1}, Sidorov proves that, for Pisot units, Property (W) is equivalent to arithmetical coding (with fundamental homoclinic point) being a.e. 1-1.

 Interest in nonperiodic tilings and spectral properties of substitutions was stimulated by Schectman's discovery, in the early 80's, of quasicrystalline materials. Self-similar point sets that have good diffractive properties, and so provide mathematical models of quasicrystals, can be created with substitutions of Pisot type: if one places an `atom' at a characteristic position in each tile of a Pisot substitution tiling (of whatever dimension) the diffraction of the resulting arrangement will have Bragg peaks. In fact, the diffraction spectrum of the arrangement will be pure point (as for a perfect quasicrystal) if and only if the tiling dynamical system has pure discrete spectrum, meaning that the eigenfunctions of the 
dynamical system span the space $L^2$ of square-integrable functions on the tiling space (\cite{D},\cite{LMS}, and see \cite{Le} for a survey). For substitution tiling dynamical systems, it is a consequence of the Halmos- von Neumann theory (and Solomyak's result that eigenfunctions are continuous (\cite{S4})) that pure discrete spectrum of the tiling dynamical system is equivalent to the system being an almost everywhere one-to-one extension of a
translation action on a torus or solenoid (see \cite{BK} for more detail and related characterizations). 

For more on Pisot substitutions and the Pisot Substitution Conjectures, see the surveys \cite{BS} and \cite{ABBLS}. A general introduction to the dynamical and topological properties of substitution tiling spaces can be found in \cite{S2}, \cite{Ro}, and  \cite{AP}. The direct antecedent to the current article is \cite{B1} in which the author dealt with $\beta$-substitutions for $\beta$ a Pisot simple Parry number. As outlined below, we extend the approach taken there to all Pisot numbers. In the summary that follows, $\theta$ is an arbitrary (primitive, nonperiodic) substitution with Pisot inflation $\beta$, $\pb$ is the $\beta$-substitution, and $\Theta$ denotes the substitution-induced homeomorphism of the substitution tiling space $\Omega_{\theta}$. Here are the key steps in our approach:

\begin{enumerate}

\item There is an almost everywhere $cr$-to-1 map $\pi_{max}:\Omega_{\theta}\to \hat{\T}^d$ that factors the $\R$-action on 
$\Omega_{\theta}$ onto a Kronecker action on a $d$-dimensional solenoid $\hat{\T}^d$. Here $d$ is the algebraic degree of $\beta$ and $cr=cr(\theta)<\infty$ is the coincidence rank of the substitution $\theta$.  \footnote{The coincidence rank is equal to the degree of the `Galois dual' shingling, or multi-tiling, of $\R^{d-1}$ mentioned previously.}

\item The system $(\Omega_{\theta},\R)$ has pure discrete spectrum if and only if $cr(\theta)=1$.

\item The structure relation for $\pi_{max}$ is strong regional proximality ($\sim_{srp}$): For $T,T'\in\Omega_{\theta}$, $\pi_{max}(T)=\pi_{max}(T')$ if and only if $T\sim_{srp}T'$. Furthermore, there are $T_1,\ldots,T_r\in\Omega_{\theta}$ with $T_i\sim_{srp}T_j$ and $T_i\cap T_j=\emptyset$ for $i\ne j$ if and only if $r\le cr(\theta)$.

\item The relation $\sim_s$, defined on $\Omega_{\theta}$ by $T\sim_s T'$ if $d(\Theta^k(T),\Theta^k(T'))\to0$ as $k\to\infty$ is open in the sense that $T\sim_s T'\implies T-t\sim_s T'-t$ for $|t|$ small. 

\item If $\tau$ and $\tau'$ are tiles, let $\tau\sim_s\tau'$ mean 
$T\sim_s T'$ for all $T,T'\in\Omega_{\theta}$ with $\tau\in T, \tau'\in T'$. Then $cr(\theta)=1$ if and only if there is a tile $\tau$ (equivalently, for all tiles $\tau$) and there are $t_i\ne0$ with $t_i\to 0$ as $i\to\infty$ so that, for each $i$, $\tau-t_i-t\sim_s \tau-t$ for a dense set of $t\in spt(\tau)\cap spt(\tau-t_i)$.

\item The stable equivalence relation ($\approx_s$) is defined on a tiling space $\Omega_{\theta}$ by $T\approx_s T'$ if and only if $T\sim_{srp} T'$ and $T-t\sim_s T'-t$ for a set of $t$ dense in $\R$. Then $cr(\theta)=1$ if and only if $\approx_s=\sim_{srp}$. If $cr(\theta)>1$, then $(\Omega_{\theta}/\approx_s)\simeq\Omega_{\tilde{\theta}}$ is a substitution tiling space for a substitution $\tilde{\theta}$, $cr(\tilde{\theta})=cr(\theta)$, and $\approx_s$ is trivial on $\Omega_{\tilde{\theta}}$.

\item There is some leeway in choosing a substitution to generate an isomorphism class of tiling spaces. If $cr(\pb)>1$, the substitution $\tpb$ may be chosen to have the special property: there is an $n\in\N$ and a letter $b$ in the alphabet
$\mathcal{A}_{\tpb}$ for $\tpb$ so that $\tpb^n(a)=b\cdots a$ for all $a\in \mathcal{A}_{\tpb}$.

\item If $\theta$ is any substitution for which there are $n\in\N$ and $b\in\mathcal{A}_{\theta}$ with $\theta^n(a)=b\cdots a$ for all $a\in\mathcal{A}_{\theta}$, then $\approx_s$ is nontrivial on $\Omega_{\theta}$.

\end{enumerate}

Pure discrete spectrum for $(\Ob,\R)$ then follows: If the spectrum were not pure discrete, the substitution $\tpb$ would be such that $\approx_s$ is trivial on $\Omega_{\tpb}$ (by (6)). But this is contradicted by (7) and (8). Versions of items (1)-(6) hold quite generally for substitutions of Pisot type in any dimension. For (1)-(3) see \cite{BKw,BBK,BK,B2}, (4) comes from \cite{BO}, a version of (5) first appears in \cite{BBK} and in the form stated here in \cite{B1} with a proof derived from \cite{BSW}, and (6) is from \cite{B1}. Item (8) is from \cite{B1}.
Section \ref{proof} is devoted to proving item (7) for the particular case of $\beta$-substitutions. The argument is somewhat technical and is based on the topological structure of $\Ob$, as we describe now. 

There are tilings in $\Ob$, corresponding to the fixed point 0 of $T_{\beta}$ and to the point of the periodic orbit, of period $p$, of $T_{\beta}$ on which the $T_{\beta}$-orbit of 1 eventually lands. These tilings are periodic under $\Pb$ of period $p$ and pairwise strong regionally proximal. Using the basic technique of \cite{BD2}, and the special `monotonic' nature of $\pb$, we show that these tilings are entangled: for $T,T'$ any two of them $T-t\sim_s T'-t$ for a set of $t$ dense in $\R^+$. But also, for such $T\ne T'$, we show that $T$ and $T'$ are not asymptotic on $\R^+$ (there are arbitrarily large $t$ so that the tiles of $T$ and $T'$ at $t$ are distinct). From this it follows (again using the monotonicity of $\pb$) that there is $\epsilon>0$ so that if $\tau$ and $\tau'$ are any two tiles having the same initial vertex $x$, then $\tau-t\sim_s \tau'-t$ for a set of $t$ dense in $[x,x+\epsilon)$, and there are tiles $\tau^1_-,\ldots,\tau^p_-$ so that if $\tau$ is any tile, there is exactly one of the $\tau^i_-$ so that, with $t_0$ chosen such that the terminal vertices of $\tau-t_0$ and $\tau^i_-$ are both equal to $y$,  $\tau-t_0-t\sim_s\tau^i_--t$ for a set of $t$ dense in $(y-\epsilon,y]$.
Item (7) then follows.

As mentioned earlier, Akiyama (in \cite{A2}) established the equivalence of the arithmetical Property (W) with pure discrete spectrum of $(\Ob,\R)$, and Sidorov (in \cite{Si1}) proved the equivalence of Property (W) with a.e.1-1-ness of arithmetical coding (at least for Pisot units). In the final Section \ref{connections} we reprove these results (and extend the Sidorov result to non-units) from a unified point of view. We think it is of interest to note that arithmetical coding occurs in the context of hyperbolic dynamics, while the map onto the maximal equicontinuous factor is determined by properties of the translation action. We will tie these two viewpoints together by showing the structure relation for the maximal equicontinuous factor map (strong regional proximality) is the same as the `global shadowing' relation defined in terms of the the hyperbolic homeomorphism $\Pb$ (such an equivalence holds also in higher dimensions - see \cite{BG}). 

Finally, Property (W) will be interpreted as a statement about homoclinic return times. The relation between the structure of these return times and pure discreteness of the translation action is a general phenomenon (see, for example, Corollary 4.5 of \cite{B2}).

\section{ Background on $\beta$-numeration and substitution tilings.}\label{background}
For $\beta>1$, the $\beta$-transformation, $T_{\beta}:[0,1]\to[0,1]$ is given by 
$$T_{\beta}(x):=\beta x-\lfloor\beta x\rfloor,$$
where $\lfloor \,\rfloor$ denotes the greatest integer function. If $\beta$ is a Pisot number, then the $T_{\beta}$-orbit of 1, $\mathcal{O}_{T_{\beta}}(1):=\{T_{\beta}^n(1):n\ge0\}$, is finite (\cite{B-M},\cite{Sch1}). Such a $\beta$ is a {\em simple Parry number} if $0\in\mathcal{O}_{T_{\beta}}(1)$ and otherwise is a {\em non-simple Parry number}. Every $x\in [0,1)$ has an {\em itenerary}
$$\kappa(x)=x_1x_2\cdots,$$
with $x_n:=i\in\{0,\ldots,\lfloor\beta\rfloor\}$ provided $T_{\beta}^{n-1}(x)\in [i/\beta,(i+1)/\beta)$. For such $x$ 
$$x=\sum_{n=1}^{\infty}x_n\beta^{-n}$$
is the ({\em greedy}) $\beta$-expansion of $x$.
The {\em kneading invariant} or {\em characteristic sequence} of $\beta$ is given by 
$$\kappa(1)=c_1c_2\cdots:=\lim_{x\nearrow1}\kappa(x),$$
the limit being taken in the product topology on the sequences. If $\beta$ is a simple Parry number, then $\kappa(1)$ is periodic of some (least) period $p$: $\kappa(1)=\overline{c_1\cdots c_p}$. If $\beta$ is a non-simple Parry number, then 
$\kappa(1)$ is strictly preperiodic: there are $m\ge1$ and $p\ge 1$ with $\kappa(1)=c_1\cdots c_m\overline{c_{m+1}\cdots c_{m+p}}$, where we take $m$, and then $p$, as small as possible.

The {\em one-sided $\beta$-shift} is the Cantor dynamical system
$$(X_{\beta}^+,\sigma),$$
with $X_{\beta^+}:=cl(\{(x_1,x_2,\ldots):x_1x_2\cdots=\kappa(x) \text{ for some }x\in [0,1]\})$ and $\sigma((x_1,x_2,  x_3,\ldots)):=(x_2,x_3,\ldots)$. One has that  $X_{\beta}$ consists precisely of the sequences of digits from $\{0,\ldots,\lfloor\beta\rfloor\}$ all of whose shifts are lexicographically no larger than the sequence $(c_1,c_2,\ldots)$ corresponding to $\kappa(1)$ (\cite{Pa}). The (two-sided) {\em $\beta$-shift} is obtained by taking inverse limits: $X_{\beta}:=\{(\ldots,x_{-1},x_0,x_1,\ldots):(x_k,x_{k+1},\ldots)\in X_{\beta}^+ \text{ for all }k\in\Z\}$. For each $x\in\R^+$ there is a corresponding point
$$\underline{x}:=(\ldots,0,0,x_{-k},\ldots,x_{-1}x_0,x_1,\ldots)\in X_{\beta}$$
where $x=\sum_{i=-k}^{\infty}x_i\beta^{-i}$, and $\sum_{i=1}^{\infty} x_{-k+i}\beta^{-i}$ is the greedy $\beta$-expansion of $\beta^{-k}x\in [0,1)$. For more detail on the $\beta$-shift, see \cite{R, Pa, IT, Bl, LM}.

For a given Pisot $\beta$, let $\{0,1\}\cup\{T_{\beta}^i(1):i=1,\ldots,m+p\}=\{0=s_0<s_1<\cdots<s_{m+p}=1\}$. The {\em prototiles} are the marked intervals $\tau_i:=[s_{i-1},s_i]\times\{i\}$, $i=1,\ldots,m+p$; their  {\em vertices} are $min(\tau_i):=s_{i-1}$ and $max(\tau_i)=s_i$. The {\em  support} of $\tau_i$ is $spt(\tau_i):=[s_{i-1},s_i]$.
We call $\tau=\tau_i+t:=(spt(\tau_i)+t)\times\{i\}$  a tile of {\em type} $i$ with support $spt(\tau_i)+t$ and vertices $min(\tau):=s_{i-1}+t$ and $max(\tau):=s_i+t$. The {\em $\beta$-substitution} on the alphabet $\mathcal{A}=\{1,\ldots,m+p\}$ is the map 
$\pb:\mathcal{A}\to\mathcal{A}^*$, $\mathcal{A}^*$ being the collection of finite nonempty words on $\mathcal{A}$, given by $\pb(i)=i_1i_2\cdots i_n$ provided
$T_{\beta}(x)$ passes through the support of $\tau_{i_1}$, then the support of $\tau_{i_2}$,..., then the support of $\tau_{i_n}$ as $x$ increases from $s_{i-1}$ to $s_i$. \footnote{Typically, the prototiles are taken to be $[0,s_i]\times\{i\}$, resulting in a different `$\beta$-substitution'. It is straightforward to pass from one to the other by a rewriting procedure (see \cite{D} or \cite{BD1}) and the resulting tiling dynamical systems are isomorphic.} The corresponding {\em tile substitution} $\Pb$ is defined on prototiles by
$\Pb(\tau_i):=\{\tau_{i_1}-min(\tau_{i_1})+\beta min(\tau_i), \tau_{i_2}-min(\tau_{i_2})+max(\tau_{i_1})-min(\tau{i_1})+\beta min(\tau_i),\ldots,\tau_{i_n}-min(\tau_{i_n})+(\sum_{j=1}^{n-1}((max(\tau_j)-min(\tau_j)))+\beta min(\tau_i)\}$. (That just says that $\Pb$ multiplies the support of $\tau_i$ by a factor of $\beta$ and tiles the resulting interval by the $\tau_{i_j}$, following the pattern of $\pb$.) The tile substitution $\Pb$ is then defined on arbitrary tiles by $\Pb(\tau_i+t):=\Pb(\tau_i)+\beta t$. By a {\em patch} $P$ we will mean a finite collection of tiles with the properties: if $\tau\ne\tau'\in P$, then $int(spt(\tau))\cap int(spt(\tau'))=\emptyset$; and $spt(P):=\cup_{\tau\in P}spt(\tau)$ is connected. $\Pb$ extends to patches by $\Pb(P):=\cup_{\tau\in P}\Pb(\tau)$.

A {\em tiling} $T$ of $\R$ is a collection of tiles with the properties: if $\tau\ne\tau'\in T$ then $int(spt(\tau))\cap int(spt(\tau'))=\emptyset$ ; and $\cup_{\tau\in T}spt(\tau)=\R$. A patch $P$ is {\em allowed} for $\Pb$ if there is a tile $\tau=\tau_i+t$ and an $n\in\N$ so that $P\subset \Pb^n(\tau)$. For $\beta\notin\N$, the {\em tiling space} associated with $\Pb$ is the collection of tilings of $\R$, all of whose patches are allowed for $\Pb$:
$$\Ob:=\{T:\text{ every patch } P\subset T \text{ is allowed for }\Pb\}.$$

For $r\ge0$ and $T\in\Ob$, let $$B_r[T]:=\{\tau\in T:spt(\tau)\cap [-r,r]\ne\emptyset\}.$$
So, for example, $B_0[T]$ is the collection of tiles of $T$ whose supports contain the origin.

The group $\R$ acts on $\Ob$ by translation, $T-t:=\{\tau-t:\tau\in T\}$, and there is a metric $d$ (called the {\em tiling metric}) on $\Ob$ with the property that $d(T,T' )$ is small if there are $r,t,t'$ with $r$ large and $|t|,|t'|$ small, so that $B_r[T-t]=B_r[T'-t']$ (see, for example, \cite{AP}). With the topology induced by $d$, $\Ob$ is compact and the translation action on $\Ob$ is continuous. It is well-known that if $\beta\notin\N$, then $\pb$ is {\em non-periodic} (meaning $T-t=T$ for some $T\in\Ob$ implies $t=0$) and {\em primitive} (there are $n$ and $i$ so that all letters occur in $\pb^n(i)$) and it follows that the translation dynamical system $(\Ob,\R)$ is minimal and uniquely ergodic (\cite{S2}). We will denote the unique Borel translation invariant measure on $\Ob$ by $\mu$. Moreover, $\Pb:\Ob\to\Ob$ by $\Pb(T):=\cup_{\tau\in T}\Pb(\tau)$ is a homeomorphism with $\mu$ invariant and ergodic (\cite{S3}). 
Note that the translation- and $\Pb$- dynamics are intertwined by:
$$\Pb(T-t)=\Pb(T)-\beta t.$$

If $\beta=n$ is an integer, the substitution $\pb$ is the periodic substitution $1\mapsto 1^n$ and, as usually defined, the tiling space is the circle $\T^1$ and the $\Psi_n$ is the $n$-fold covering map $x+\Z\mapsto nx+\Z$. It will be convenient
(looking ahead to the discussion of arithmetical coding) to instead define $\Omega_{\psi_n}$ to be the solenoid $\inv \Psi_n$, interpret $\Psi_n$ to be the shift homeomorphism, and take the $\R$-action to be the natural Kronecker action $(x_1+\Z,x_2+\Z,\ldots)-t:=(x_1-t+\Z,x_2-t/n+\Z,\ldots)$.

A continuous map $f:\Ob\to\T^1:=\R/\Z$ is a {\em continuous eigenfunction} of $(\Ob,\R)$ if there is $\gamma\in\R$ (the associated eigenvalue) so that $f(T-t)=f(T)-(\gamma t+\Z)$ for all $t\in\R$. The system $(\Ob,\R)$ is said to have {\em pure discrete spectrum} if the continuous eigenfunctions span $L^2(\mu)$. An alternative formulation of this property, in terms of equicontinuous factors, is closer to the spirit of this article. A continuous system $(Y,\R)$ (that is, a continuous action of $\R$ on the metric space $Y$) is {\em equicontinuous} provided the collecton of homeomorphisms $\{y\mapsto y\cdot t\}_{t\in\R}$ is equicontinuous. A {\em maximal equicontinuous factor} of a system $(X,\R)$ is an equicontinuous factor $(X_{max},\R)$ of $(X,\R)$ with the property that every equicontinuous factor of $(X,\R)$ is also a factor of $(X_{max},\R)$. Compact minimal systems always have a (unique up to isomorphism) maximal equicontinuous factor
(\cite{Aus}). In the case at hand, the maximal equicontinuous factor of $(\Ob,\R)$ is a Kronecker action on a $d$-dimensional torus or solenoid $\hat{\T}_{\beta}^d$, with $d$ the algebraic degree of $\beta$ (\cite{BBK},\cite{B2}). Here $\hat{\T}_{\beta}^d$ is the inverse limit $\hat{\T}_{\beta}^d:=\inv F_M$ with $M$ the companion matrix of the minimal polynomial of $\beta$ and $F_M:\T^d:=\R^d/\Z^d\to\T^d$ the hyperbolic toral endomorphism $F_M(v+\Z^d):=Mv+\Z^d$ and the Kronecker action is given by $(v_1+\Z^d,v_2+\Z^d,v_3+\Z^d,\ldots)-t:=(v_1-t\omega+\Z^d,v_2-\beta^{-1}t\omega,v_3-\beta^{-2}t\omega+\Z^d,\ldots)$, where $\omega$ is a right positive eigenvector for the eigenvalue $\beta$ of $M$.

The structure relation for the maximal equicontinuous factor map $\pi_{max}:\Ob\to\hat{\T}^d_{\beta}$ is given by strong regional proximality \footnote{In the general context of an abelian group acting minimally on a compact metric space, the structure relation for the maximal equicontinuous factor map is regional proximality (\cite{V}). For Pisot substitution tiling spaces, the Meyer property is responsible for the strong version we use here.} (\cite{BK}): Tilings $T,T'$ are {\em strong regionally proximal}, denoted $T\sim_{srp}T'$, provided for all $R>0$ there are $S_R,S'_R\in\Ob$ abd $t_R\in\R$ so that $$B_R[T]=B_r[S_R],$$
$$B_R[T']=B_r[S'_R],\text{ and}$$
$$B_R[S_R-t_R]=B_r[S'_R-t_R].$$

There is $r\in\N$, called the {\em coincidence rank} of $\pb$ so that $\pi_{max}:\Ob\to\hat{T}^d_{\beta}$ is a.e. $r$-to-1
and for each $z\in\hat{T}^d_{\beta}$ there are $S_1,\ldots,S_r\in\pi_{max}^{-1}(z)$ with $S_i\cap S_j=\emptyset$ for $i\ne j$. The system $(\Ob,\R)$ has pure discrete spectrum if and only if the coincidence rank of $\pb$ equals 1, that is, if and only if $\pi_{max}$ is a.e. 1-1 (see \cite{BKw, BBK, BK}).

\begin{conjecture} \label{Pisot Conjecture} (Pisot conjecture for $\beta$-substitutions) If $\beta$ is a Pisot number, then $(\Ob,\R)$ has pure discrete spectrum.
\end{conjecture}

\section{ Arithmetical coding and Property (W)}\label{coding and W}

Given a Pisot number $\beta$ of algebraic degree $d$ and $M$ the companion matrix of its minimal polynomial, let $E^u$ and $E^s$ denote the 1-dimensional and $(d-1)$-dimensional unstable and stable spaces of $M$. (So $E^u$ is spanned by the positive right eigenvector $\omega$ and $E^s=\{v\in\R^d:M^nv\to0 \text{ as }n\to\infty\}$.) Then $\R^d=E^u\oplus E^s$ and for $y\in\R^d$ we write $y=y^u+y^s$ with $y^u\in E^u$ and $y^s\in E^s$. A {homoclinic point} $\bar{y}=y+\Z^d\in\T^d$, by which we mean $y=z^u$ for some $z\in\Z^d$, is {\em fundamental} if $\{z,Mz,\ldots,M^{d-1}z\}$ is a basis for $\Z^d$. (Such a point always exists.) If $\bar{y}$ is a fundamental homoclinic point, let $\hat{\bar{y}}:=(y+\Z^d, M^{-1}y+\Z^d,\ldots)\in\hat{\T}^d_{\beta}$ be the corresponding point, homoclinic to $\hat{\bar{0}}$ under the shift automorphism $\hat{F}_M$. The map $$h_{\bar{y}}: X_{\beta}\to\hat{\T}^d_{\beta}$$
given by $$h_{\bar{y}}((x_i)):=\sum_{i\in\Z}x_i\hat{F}_M^{-i}(\hat{\bar{y}})$$ 
is called an {\em aritmetical coding} of  the hyperbolic automorphism $\hat{F}_M:\hat{\T}^d_{\beta}\to\hat{\T}^d_{\beta}$.
Such a coding has the felicitous properties: 
$$h_{\bar{y}}(\underline{\beta x})=\hat{F}_M(h_{\bar{y}}(\underline{x}))$$
for all $x\in\R^+$, more generally
$$h_{\bar{y}}\circ \sigma=\hat{F}_M\circ h_{\bar{y}},$$ and $$h_{\bar{y}}(\underline{x+x'})=h_{\bar{y}}(\underline{x})+h_{\bar{y}}(\underline{x'})$$
for all $x,x'\in\R^+$. 


The point $x\in \R^+$ is said to have {\em finite $\beta$-expansion} if $\underline{x}=(\ldots,0,x_{-k},\ldots,x_n,0,0,\ldots)$ for some $k\in\N$. Let $Fin(\beta):=\{x\in\R^+: x\text{ has finite }\beta\text{-expansion}\}$. The Pisot number $\beta$ is said to be {\em finitary} if $\Z[1/\beta]\cap\R^+\subset Fin(\beta)$ and is said to be {\em weakly finitary} if for all $z\in\Z[1/\beta]\cap\R^+$ there are $x,y\in Fin(\beta)$, with $y$ as small as desired, such that $z=x-y$. The so-called {\em Property (W)} is that $\beta$ is weakly finitary.


Akiyama proves in \cite{A1} that $\beta$ satisfies Property (W) if and ony if $(\Ob,\R)$ has pure discrete spectrum and Sidorov (\cite{Si1}) shows that, at least for Pisot $\beta$ that are algebraic units, Property (W) is equivalent to arithmetical coding (with fundamental homoclinic point) being a.e. 1-1. It is shown in \cite{ARS} that all cubic Pisot units, and higher degree Pisot numbers satisfying a `dominant condition' are weakly finitary.  We prove Conjecture \ref{Pisot Conjecture} in the next section and provide an alternative argument for the equivalence of Conjecture \ref{Pisot Conjecture} with a.e. one-to-oneness of arithmetical coding for Pisot $\beta$ and satisfaction of Property (W) for all Pisot $\beta$ in the final section.
(In fact we will prove the stronger form of Property (W): If $\beta$ is Pisot and $z\in\Z[1/\beta]\cap\R^+$, then $\{y:y\in Fin(\beta) \text{ and }y+z\in Fin(\beta)\}$ is dense in $\R^+$.)

\section{ Proof of the Pisot Conjecture for $\beta$-substitutions.} \label{proof}

The {\em language} of $\pb$ is defined by $\mathcal{L}(\pb):=\{w\in\mathcal{A}^*: w \text{ is a factor of }\pb^n(i) \text{ for some }n\in\N, i\in\mathcal{A}\}$.
  
The following `monotonicity'  properties of $\mathcal{L}(\pb)$ will be used repeatedly. 
\\
\\
{\bf Property 1:} If $ab$ is a two letter word in $\mathcal{L}(\pb)$ and $b\in\{2,\ldots,m+p\}$, then $a=b-1$.
\\
\\
{\bf Property 2:} If $ab$ and $ac$ are two letter words in $\mathcal{L}(\pb)$ and $b\ne c$, then either $b=1$ or $c=1$. 
\\
\\
{\bf Property 3:} If $ac$ and $bc$ are two letter words in $\mathcal{L}(\pb)$ and $a\ne b$, then $c=1$. 
\\
\\
The first property is a simple consequence of the fact that $T_{\beta}$ increases monotonically from 0 on each of its maximal intervals of continuity. Properties 2 and 3 follow immediately from Property 1. Note that Property 1 implies that if $\tau_k+t\in T\in\Ob$, then $\{\tau_1,\tau_2,\ldots,\tau_k\}+t\subset T$.

We have labeled the prototiles according to the order of their occurance on the interval $[0,1]$. It will be convenient to also have labelings of the prototiles according to where, along the orbit of $1$ under $T_{\beta}$, their maximum
and minimum vertices lie: Let $z^i:=T_{\beta}^{i-1}(1)$ for $i=1,\ldots,m+p$; then if $max(\tau_i)=z^j$, define $\tau_-^j:=\tau_i$, and if $min(\tau_i)=z^j$ set $\tau_+^j:=\tau_i$. If $m>0$, then corresponding to each of the $T_{\beta}$-periodic points $z^{m+i}$, $i=1,\ldots,p$, there are two $\Pb$-periodic tilings:
$$T_i:=\cup_{n=0}^{\infty}\Pb^{np}(\{\tau_-^{i+m}-z^{m+i},\tau_+^{m+i}-z^{m+i}\})$$ and
$$T_i^0:=\cup_{n=0}^{\infty}\Pb^{np}(\{\tau_-^{m+i}-z^{m+i},\tau_1\}).$$

Note that $T_j^0\sim_{srp}T_i^0$ and $T_k^0\sim_{srp}T_k$ for all $i,j,k\in\{1,\ldots,p\}$, so also $T_i\sim_{srp} T_j$ for all $i,j\in\{1,\ldots,p\}$.

Given $T,T'\in \Ob$ we write $T\sim_s T'$ provided $d(\Pb^k(T),\Pb^k(T'))\to 0$ as $k\to\infty$. Since $d(\Pb^k(T),\Pb^k(T'))\to 0$ as $k\to\infty$ if and only if there is $k\in\N$ so that $B_0[\Pb^k(T)]=B_0[\Pb^k(T')]$ (\cite{BO}), we see that $T\sim_s T'\implies T-t\sim_s T'-t$ for all $|t|<\epsilon$, for some $\epsilon>0$. We'll say that $T\sim_s T'$ {\em densely on the interval $J$} if $T-t\sim_s T'-t$ for a dense set of $t\in J$. Note that then $T-t\sim_s T'-t$ for an open dense subset of $t$ in $J$. If $P,P'$ are allowed patches, we'll say that $P\sim_s P'$ densely on $J$ if $T\sim_s T'$ densely on $J$ for all $T,T'\in\Ob$ with $P\subset T$, $P'\subset T'$.

\begin{lemma} \label{stably equiv 1}Suppose that $m>0$ and $p>1$. There are then $i\ne j$ so that $T_i\sim_s T_j$ densely on $\R^+$. If $m>0$ and $p=1$, then $T_1\sim_sT_1^0$ densely on $\R^+$.
\end{lemma}

\begin{proof} We give a variant of the basic argument of \cite{BD2}. Let $r$ be the coincidence rank of $\pb$. We may assume that $r>1$, otherwise $T_i\sim_s T_j\sim_sT_k^0$ densely on $\R$ for all $i,j,k$. ($T_i\sim_{srp}T_j\sim_{srp} T_k^0$ for all $i,j,k$, as noted above, and, in the context of any Pisot substitution tiling, $T\sim_{srp}T'\implies T\sim_sT'$ densely on $\R$ - see \cite{BK}.) There are tilings $S_1,\ldots,S_r$ with the properties (see \cite{BK}):
\begin{enumerate}
\item $S_i\sim_{srp} S_j$ for $1\le i,j\le r$,
\item $S_i$ is $\Pb$-periodic for each $i\in\{1,\ldots,,r\}$, and
\item $S_i\cap S_j=\emptyset$ for all $i\ne j\in\{1,\ldots,r\}$.
\end{enumerate}
Replacing $\Pb$ by an appropriate power, we may assume that the $S_i$ are all fixed by $\Pb$. For each $i$, let $V_i$ be the collection of all the vertices of tiles in $S_i$ and let $V=\{\cdots v_{-1}<v_1<v_2\cdots\}:=\cup_{i=1}^rV_i$. For each $j\in\Z$ and $i\in\{1,\ldots,r\}$, let $\eta_j^i$ be the tile of $S_i$ with $(v_j+v_{j+1})/2\in spt(\eta_j^i)$. We call the collection $\mathcal{C}_j:=\{\eta_j^i:i=1,\ldots,r\}$ a {\em configuration}. For each $j\in\Z$, let $m_j:=max\{min(\eta_j^i):i=1,\ldots,r\}$ and $M_j:=min\{max(\eta_j^i):i=1,\ldots,r\}$.

Up to translation, there are only finitely many configurations (this is the Meyer property, see \cite{BK}). Thus, since the $S_i$ are not translation-periodic, there are $k,l\in\Z$ and $w\in\R$ with 
$\mathcal{C}_k=\mathcal{C}_l+w$ but $\mathcal{C}_{k+1}\ne\mathcal{C}_{l+1}+w$. For each $i\in\{1,\ldots,r\}$, let $i'\in\{1,\ldots,r\}$ be (uniquely) defined by $\eta_k^{i'}=\eta_l^i+w$.
Note that if $max(\eta_j^i)>M_j$, then $\eta_{j+1}^i=\eta_j^i$. Hence, letting $\mathcal{C}_l':=\{\eta_l^i: \eta_{k+1}^{i'}\ne\eta_{l+1}^i\}$ we see that $max(\eta_l^i)=M_l^i$ for all $\eta_l^i\in\mathcal{C}_l'$.
Furthermore, if $\eta_l^i\in\mathcal{C}_l'$, then either $\eta_{l+1}^i$ is of type 1 (that is, is a translate of $\tau_1$) or $\eta_{k+1}^{i'}$ is of type 1 (this is Property 2). Since the $S_i$ are pairwise disjoint, it follows that $\sharp\mathcal{C}_l'\in\{1,2\}$.

We consider two cases:

{\bf Case 1:} $\sharp\mathcal{C}_l'=1$.

Let $\{\eta_l^i\}=\mathcal{C}_l'$. We claim that $\eta_{l+1}^i+w\sim_s \eta_{k+1}^{i'}$ densely on $[0,\epsilon)$, where  $\epsilon:=min\{M_{l+1}-m_{l+1},M_{k+1}-m_{k+1}\}$. If not, there is $t_0\in (0,\epsilon)$ and $\delta>0$ so that $B_0[\Pb^n(S_{i}-t)]\cap B_0[\Pb^n(S_{i'}-w-t)]=\emptyset$ for all $n\in\N$ and $|t-t_0|<\delta$. That is, $B_{\beta^n\delta/2}[\Pb^n(S_i-t_0)]\cap B_{\beta^n\delta/2}[\Pb^n(S_{i'}-w-t_0]=\emptyset$ for all $n\in\N$. Choose $n_k\to\infty$ so that $\Pb^{n_k}(S_j-t_0)\to W_j\in \Ob$ and $\Pb^{n_k}(S_j+w-t_0)\to U_j\in\Ob$ as $k\to\infty$, for all $j=1,\ldots,r$ (such $\{n_k\}$ exists by compactness of $\Ob$). Then the $W_j$, $j=1,\ldots,r$, are pairwise disjoint (the $S_j$ are pairwise disjoint and fixed by $\Pb$) and stongly regionally proximal ($\sim_{srp}$ is closed and preserved by both translation and $\Pb$), as are the $U_j$. Furthermore, $W_i\cap U_{i'}=\emptyset$.
Since $\emptyset\ne\{W_j:j\ne i\}=\{U_j:j\ne i'\}$, $W_i\sim_{srp} U_{i'}$. We now have $r+1$ pairwise disjoint and strongly regionally proximal tilings ($W_1,\ldots,W_r, U_{i'}$), contradicting coincidence rank $r$.

So $\eta_{l+1}^i+w\sim_s \eta_{k+1}^{i'}$ densely on $[0,\epsilon)$. Again, since $\eta_l^{i}+w=\eta_k^{i'}$, it follows from Property 2 that one of $\eta_{l+1}^{i}$ and $\eta_{k+1}^{i'}$ must be of type 1: say $\eta_{l+1}^{i}=\tau_1+M_l$ and $\eta_{k+1}^{i'}=\tau_j-min(\tau_j)+M_l+w$. We have that $\tau_1\sim_s\tau_j-min(\tau_j)$ densely on $[0,\epsilon)$ (and $j\in\{2,\ldots,m+p\}$ since $\mathcal{C}_{k+1}\ne \mathcal{C}_{l+1}+w$). Let 
$q\in\{1,\ldots,p\}$ be such that $q$ and $j$ are in phase, that is, $p|(q-j)$. Then $\tau_1\sim_s\tau_q-min(\tau_q)$ densely on $[0,\epsilon)$. Since $\Pb^p(\tau_1)=\{\tau_1,\ldots\}$ and $\Pb^p(\tau_q-min(\tau_q))=\{\tau_q-min(\tau_q),\ldots\}$ we see that $T_n^0\sim_sT_q$ densely on $\R^+$ (for all $n$). Then $T_j=\Pb^{j-q}(T_q)\sim_s\Pb^{j-q}(T_n^0)=T_{n+j-q}^0\sim_sT_q$ densely on $\R^+$ for all $j=1,\ldots,p$.

{\bf Case 2:} $\sharp\mathcal{C}_l'=2$.

Say $\mathcal{C}_l'=\{\eta_l^{i_1},\eta_l^{i_2}\}$ and, without loss of generality: 
$$\eta_{l+1}^{i_1}=\tau_1+M_l,$$
$$\eta_{l+1}^{i_2}=\tau^i_+-z^i+M_l,$$
$$\eta_{k+1}^{i_1'}=\tau^j_+-z^j+M_l+w, \text{ and }$$
$$\eta_{k+1}^{i_2'}=\tau_1+M_l+w,$$
for some $1\ne i\ne j\ne1$. Suppose that $i$ and $j$ are in phase and,  say, $0<i<j$. There is then $q$ so that $i+q=m$ and $j+q=m+sp$, which we may take to be $m+p$. Now, from Property 1, $\tau^i_--z^i+M_l\in\mathcal{C}_l$
and $\tau^j_--z^j+M_l+w\in\mathcal{C}_k=\mathcal{C}_l+w$. Thus $\tau^j_--z^j+M_l\in\mathcal{C}_l$. Say $\tau^i_--z^i+M_l\in S_a$ and $\tau^j_--z^j+M_l\in S_b$. Then $\tau^{m+1}_--z^{m+1}+\beta^{q+1}M_l\in\Pb^{q+1}(S_a)\cap\Pb^{q+1}(S_b)=S_a\cap S_b$, in contradiction to the disjointness of $S_a$ and $S_b$.

Thus, $i$ and $j$ are not in phase (so $p>1$). It follows as in Case 1 that $\tau^i_+-z^i\sim_s\tau^j_+-z^j$ densely on $[0,\epsilon)$ for some $\epsilon>0$ and then that $T_{i'}\sim_s T_{j'}$ densely on $\R^+$, with $i'\ne j'\in\{1,\ldots,p\}$ such that $i'+m$ is in phase with $i$ and $j'+m$ is in phase with $j$. 
\end{proof}

Tilings $T,T'$ are {\em asymptotic on $\R^+$} provided $\lim_{t\to\infty}d(T-t,T'-t)=0$. It follows from the Meyer property, which all Pisot substitution tilings enjoy, that $T,T'\in\Ob$ are asymptotic on $\R^+$ if and only if there is $t_0$ so that $B_0[T-t]=B_0[T'-t]$ for all $t>t_0$ (see \cite{BO}). For a discussion of the fundamental role asymptotic tilings play in the classification of one-dimensional substitution tilings, and an algorithm for finding them, see \cite{BD1}.

The tilings $T_j^0$ are all pairwise asymptotic on $\R^+$. At least for some $\beta$, $\Ob$ has additional asymptotic pairs.

\begin{example} \label{golden mean} Consider $\beta=(3+\sqrt{5})/2$. Then $\kappa(1)=2\overline{1}$ and the substitution is $$1\to121;\,\,\,2\to 21.$$
The tilings $T_1^0$ and $T_1$ are asymptotic on $\R^+$.
\end{example}

\begin{example} \label{asymptotic example} Consider $\beta$ with $\kappa(1)=22\overline{01}$ and substitution 
$$1\to 12;\,\,\,2\to34;\,\,\,3\to2341;\,\,\,4\to23.$$
The tilings corresponding to the $T_{\beta}$-periodic points with iteneraries $\overline{21}$ and $\overline{20}$ are asymptotic on $\R^+$.

\end{example}

For an arbitrary substitution $\phi$ with prototiles $\rho_1,\ldots,\rho_d$, the rose $\mathcal{R}_{\phi}$ is the wedge of circles obtained by identifying the vertices in the disjoint union of the prototiles: $\mathcal{R}_{\phi}=\cup_{i=1}^d\rho_i/\sim$.
Let $g_{\phi}:\Omega_{\phi}\to\mathcal{R}_{\phi}$ be given by 

\begin{equation}\label{def of g}
g_{\phi}(T)=[(x,i)]
\end{equation}
provided $x\in spt(\rho_i)$ and $\rho_i-x\in T$, and let 

\begin{equation} \label{def of f}
f_{\phi}:\mathcal{R}_{\phi}\to\mathcal{R}_{\phi}
\end{equation}
be the continuous map induced by $\Phi$; that is, $f_{\phi}\circ g_{\phi}=g_{\phi}\circ\Phi$. For $\beta$ substitutions, it will be convenient to also have the (discontinuous) map 
 
 \begin{equation} \label{def of g_I}
g_I: \Ob\to I=[0,1]
\end{equation}
given by $g_I(T)=x$ provided $min(\tau_j)\le x< max(\tau_j)$ and $\tau_j-x\in T$. Note that $T_{\beta}\circ g_I=g_I\circ\Pb$ and $g_{\pb}=q\circ g_I$, where $q:I\to \mathcal{R}_{\pb}\simeq I/\{0=s_0<s_1<\cdots<s_{m+p}=1\}$ is the natural quotient map and $\{s_i:i=0,\ldots,s_{m+p}\}=\{0,1\}\cup\mathcal{O}_{T_{\beta}}(1)$.

The following somewhat technical result is crucial to our analysis.

\begin{prop} \label{not asymptotic} If $T\ne T'\in\Ob$ are asymptotic on $\R^+$ and are both $\Pb$-periodic of (least) period $p$, then either $\{T,T'\}\subset \{T_j^0:j=1,\ldots,p\}$ or $T$ and $T'$ are in distinct $\Pb$-orbits.
\end{prop}
\begin{proof} Let's suppose that $T\ne T'\in\Ob$ are asymptotic on $\R^+$ and $\Pb$-periodic of period $p$. Let $t_0=t_0(T,T'):=\inf\{t':B_0[T-t]=B_0[T'-t] \text{ for all }t>t'\}$. Then there are $k\ne k'$ so that $B_0[T-t_0]=\{\tau_-^k-z^k+t_0,\tau_1+t_0\}$ and $B_0[T'-t_0]=\{\tau_-^{k'}-z^{k'}+t_0,\tau_1+t_0\}$. (The occurance of the $\tau_1$ in both patches follows from Property 3.) If $t_0(T,T')=0$, then $\{T,T'\}\subset\{T_j^0:j=1,\ldots,p\}$. We will assume from now on that $t_0(T,T')>0$. Then $k,k'$ must be {\em in phase}; that is, $p|(k-k')$, for otherwise, $t_0(T,T')=t_0(\Pb^p(T),\Pb^p(T'))=\beta^p t_0(T,T')> t_0(T,T')$. Since $t_{\pb}(r)=t_{\pb}(s)$ for $r\ne s$ only when $max(\tau_r),\max(\tau_s)\in\{z^m,z^{m+p}\}$, $t_0(n):=t_0(\Pb^n(T),\Pb^n(T'))$ will increase with $n$ until, for some $n_0$, $\tau_-^m-z^m+t_0(n_0), \tau_-^{m+p}-z^{m+p}+t_0(n_0)\in B_0[\Pb^{n_0}(T)-t_0(n_0)]\cup B_0[\Pb^{n_0}(T')-t_0(n_0)]$. Let us examine the effect of applying $\Pb$ to the patches in $\Pb^{n_0}(T)$ and $\Pb^{n_0}(T')$ that lie over $[0,t_0(n_0)]$.

First let's suppose that $z^{m+p}<z^m$ and, without loss, that $\tau_-^{m+p}-z^{m+p}+t(n_0)\in \Pb^{n_0}(T)$ and $\tau_-^m-z^m+t_0(n_0)\in \Psi^{n_0}(T')$. Let $k,k'$ be such that $\tau_k=\tau_-^{m+p}$ and $\tau_{k'}=\tau_-^m$. From Property 1
we have that the patches $P:=\{\tau_1,\ldots\tau_k\}-z^{m+p}+t_0(n_0)$ and $P':=\{\tau_1,\ldots,\tau_{k'}\}-z^m+t_0(n_0)$ are contained in $\Psi^{n_0}(T)$ and $\Psi^{n_0}(T')$, respectively. Since $T_{\beta}(z^{m+p})=T_{\beta}(z^m)$, we have that $T_{\beta}(t+z^m-z^{m+p})=T_{\beta}(t)$ for $0\le t\le z^{m+p}$ and it follows that $\Pb(P)\subset\Pb(P')$. Let $\tau_-^l-z^l-z^{m+p}+t_0(n_0)$ be the tile in $\Pb^{n_0}(T)$ that immediately precedes $P$.
We have that $t_0(n_0+1)=\beta t_0(n_0)-\beta z^{m+p}$ and $\tau_-^{l+1},\tau_-^1\in B_0[\Pb^{n_0+1}(T)-t_0(n_0+1)]\cup B_0[\Pb^{n_0+1}(T')-t_0(n_0+1)]$. Now $t_0(n)$ must increase for $n=n_0+1,\ldots,n_0+m-1$ and this means that $m\le p$. Thus, since $l+1$ and $1$ must be in phase, $l=p$.

An identical analysis with the same conclusion applies to the case $z^m<z^{m+p}$. To ease notation, we may then (replacing $T$ and $T'$ by $\Pb^{n_0+1}(T)$ and $\Pb^{n_0+1}(T')$) assume that $B_0[T-t_0(0)]=\{\tau_-^{p+1}-z^{p+1}+t_0(0),\tau_1+t_0(0)\}$ and $B_0[T'-t_0(0)]=\{\tau_-^1-1+t_0(0),\tau_1+t_0(0)\}$. We see now that 
$B_0[\Pb^n(T)-t_0(n)]=\{\tau_-^{p+1}-z^{p+1}+t_0(n),\tau_1+t_0(n)\}$ and $B_0[\Pb^n(T')-t_0(n)]=\{\tau_-^1-1+t_0(n),\tau_1+t_0(n)\}$ if and only if $n=km$ for some $k$. Hence $m|p$.

Let us next rule out $z^m<z^{m+p}$. Suppose that this is the case. The patch $P':=\{\tau_1,\tau_2,\ldots,\tau_-^m\}-z^m+t_0(m-1)\}\subset \Pb^{m-1}(T')$ is immediately preceded (in $\Pb^{m-1}(T')$) by a translate of some tile $\tau_-^l$.
Then $B_0[\Pb^m(T')-t_0(m)]=\{\tau_-^{l+1}-z^{l+1}+t_0(m),\tau_1+t_0(m)\}$ while $B_0[\Pb^m(T)-t_0(m)]=\{\tau_-^1-1+t_0(m),\tau_1+t_0(m)\}$. Since $z^m<z^{m+p}$, $m$ must be greater than 1 and then $z^m<\beta^{m-1}z^1$. It follows that the tile $\tau_1-z^m+t_0(m-1)$ in $P'$ can't be the first tile in a patch $\Pb^{m-1}(\tau_1+\beta^{1-m}(t_0(m-1)-z^m)$ with $\tau_1+\beta^{1-m}(t_0(m-1)-z^m)\in T'$. Thus it must be the case that there is $k$, $1\le k<m-1$, and $\tau\in \Pb^k(T')$ with $\Pb(\tau)=\{\ldots,\tau_-^1-z^1+t',\tau_1+t',\ldots\}$ and $\Pb^{m-2-k}(\tau_1+t')=\{\tau_1-z^m+t_0(m-1),\ldots\}$. But then $l=m-1-k$ so that $1<l+1<m$ and then $1$ and $l+1$ are definitely out of phase. 

Thus $z^{m+p}<z^m$. 

For each $n\in\N$ we have:
\begin{enumerate}
\item $B_0[\Pb^{nm}(T)-t_0(nm)]=\{\tau_-^{p+1}-z^{p+1}+t_0(mn),\tau_1+t_0(nm)\},$ 
\item $B_0[\Pb^{nm}(T')-t_0(nm)]=\{\tau_-^1-z^1+t_0(mn),\tau_1+t_0(nm)\},$
\item $\{\tau_-^p-z^p,\tau_1,\tau_2,\ldots,\tau_-^{m+p}\}-z^{m+p}+t_0(mn-1)\subset \Pb^{mn-1}(T), \text{ and}$
\item $\{\tau_1,\tau_2,\ldots,\tau_-^m\}-z^m+t_0(mn-1)\subset \Pb^{mn-1}(T').$
\end{enumerate}

{\bf Claim 1}: $t_0(nm)<1$ for all $n$.

To see that this must be true, suppose that $b=t_0(nm)-1>0$ for some $n$. Then the patch $P_n=\{\tau_1,\tau_2,\ldots, \tau_-^1\}-z^1+t_0(nm)$ in $\Pb^{nm}(T')$ has the property that $\Pb^m(P)$ contains the patch $P_{n+1}$ in $\Pb^{(n+1)m}(T')$, so that $t_0((n+1)m)\ge\beta^m b+1$. That is, $t_0((n+1)m)-1\ge\beta^mb$. Inductively,  $t_0((n+k)n)\ge \beta^kb$ for $k\in\N$. But $t_0$ is periodic.
If $t_0(nm)=1$ for some $n$ then $\tau_1\in \Pb^{nm}(T')$ and $\tau_1\in\Pb^{km}(T')=T'$, so $T'\ne T_j$ for any $j$. We have established Claim 1.
\\
\\
Let $y:=g_I(T')$ (recall definition \ref{def of g_I}) and let $r:=\lfloor \beta z^{m+p}\rfloor$. In other words, $y=1-t_0(0)=z^1-t_0(0)$ and $r$ is the number of discontinuities of $T_{\beta}$ on $[0,z^{m+p}]$.

{\bf Claim 2}: $\kappa(y)=\overline{a_1a_2\cdots a_{m-1}(a_m-r-1)}$.

To prove the claim, we first suppose that there is $j<m$ with $(\kappa(y))_j<a_j$. Then the patch $\{\tau_1,\tau_2,\ldots,\tau_-^{j+1}\}-z^{j+1}+t_0(j)$ in $\Pb^{j}(T')$ has support contained in $[0,\infty)$. The patch $\{\tau_1,\tau_2,\ldots,\tau_-^m\}-z^m+t_0(m-1)$ of $\Pb^{m-1}(T')$ then also has support in $[0,\infty)$ and hence so does the patch $\{\tau_1,\tau_2,\ldots,\tau_-^1\}-z^1+t_0(m)$. This means that $t_0(m)\ge 1$, which has been ruled out by Claim 1.
Thus $(\kappa(y))_j=a_j$ for $j=1,\ldots,m-1$. It follows that $t_0(j)=z^{j+1}-T_{\beta}^j(y)$ for $j=0,\ldots,m-2$

If $|t_0(m-1)-t|<z^{m+p}$, then $B_0[\Pb^m(T)-\beta t]=B_0[\Pb^m(T')-\beta t]$. Consequently, $t_0(m-1)>z^{m+p}$, which means that $T_{\beta}^{m-1}(y)<z^{m}-z^{m+p}$. For $t\in [0,z^{m+p}]$, $T_{\beta}(t)=T_{\beta}(t+z^m-z^{m+p})$. Thus $T_{\beta}$ has $r$ discontinuities on the interval $(z^m-z^{m+p}, z^m]$; it's also discontinuous at $z^m-z^{m+p}$. Thus $T_{\beta}$ has at least $r+1$ discontinuities between $T_{\beta}^{m-1}(y)$ and $z^m$ and we have that $(\kappa(y))_{m}\le a_{m}-r-1$. Suppose that $(\kappa(y))_{m}<a_m-r-2$. Then  $t_0(m-1)>z^m-z^{m+p}-1/\beta$ and $t_0(m)\ge1$, contrary to Claim 1. Thus, $(\kappa(y))_j$ is as claimed for $j=1,\ldots,m$. Repeating the argument periodically gives $\kappa(y)=\overline{a_1a_2\cdots a_{m-1}(a_m-r-1)}$, as claimed.
\\
\\
From Claim 2, $y$ is $T_{\beta}$-periodic of period $m$. Thus $T'$ is $\Pb$-periodic of period $m$, so $m=p$ and we have that $\kappa(y)=\overline{a_1a_2\cdots a_{p-1}(a_p-r-1)}$. Let's now determine itenerary $\kappa(x)$ of $x:=g_I(T)$.

Consider the patch $\{\tau_-^p-z^p,\tau_1\}-z^{2p}+t_0(p-1)\subset \Pb^{p-1}(T)$ (see item (3), with $m=p$, $n=1$) and note that $\tau_1-z^{2p}+t_0(p)$ has support contained in $\R^+$. This can only come from  a patch $\{\tau_-^{p-1}-z^{p-1},\tau_1\}-\beta^{-1}(z^{2p}-t_0(p-1))$ in $\Pb^{p-2}(T)$, which can only have come from a patch $\{\tau_-^{p-2}-z^{p-2},\tau_1\}-\beta^{-2}(z^{2p}-t_0(p-1))$ in $\Pb^{p-3}(T)$,... .
We see that the patch $\{\tau_-^1-1,\tau_1,\tau_2,\ldots.\tau_-^{p+1}\}-z^{p+1}+t_0(0)$ is contained in $T$ and the support of $\tau_1-z^{p+1}+t_0(0)$ is contained in $\R^+$. The origin is then located in the interior of the support of the patch $\{\tau_1,\tau_2,\ldots,\tau_-^1\}-1-z^{p+1}+t_0(0)\subset T$ (see Claim 1) at relative position $x$. That is, $x=1+z^{p+1}-t_0(0)$. Since, under aplication of $\Pb^{p-1}$, the tile $\tau_-^1-1-z^{p+1}+t_0(0)$ produces the patch $\{\ldots,\tau_-^p-z^p-z^{2p}+t_0(p-1)\}$ in $\Pb^{p-1}(T)$ and the support of the latter patch is contained in $(-\infty, t_0(p-1)]$, we have $\beta^j(max(\tau_-^1-1-z^{p+1}+t_0(0)))=\beta^j(t_0(0)-z^{p+1})\le t_0(j)$ for $j=0,\ldots,p-1$. If there were a discontinuity of $T_{\beta}$ between $T_{\beta}^j(x)$ and $z^{j+1}$ for some $j\in\{0,\ldots,p-2\}$, then $t_0(j+1)$ would be at least 1, which is not allowed by Claim 1. Thus $(\kappa(x))_j=a_j$ for $j=1,\ldots,p-1$. 

One can show that $\kappa(x)=\overline{a_1a_2\cdots a_{p-1}(a_p-1)}$ (as in Example \ref{asymptotic example}) but at this point we have all we need to know that the tilings $T\ne T'$ are not in the same $\Pb$-orbit. Indeed, since $\Pb\circ g_I=g_I\circ T_{\beta}$, $\Pb^k(T)=T'$ would imply $\sigma^k(\kappa(g_I(T)))=\kappa(g_I(T'))$, which would mean that
$\sum_{j=1}^p(\kappa(g_I(T)))_j=\sum_{j=1}^p(\kappa(g_I(T')))_j$. The latter equality can only hold if $g_I(T)=g_I(T')$ (since we have shown that $(\kappa(g_I(T)))_j=(\kappa(g_I(T')))_j$ for $j=1,\ldots,p-1$), but then $T=T'$.

\end{proof} 

\begin{corollary} \label{cor} If $i\ne j$, then $T_i$ and $T_j$ are not asymptotic on $\R^+$.
\end{corollary}

\begin{lemma}\label{stably equiv 2} If $m>0$ then $T_i\sim_s T_j\sim_s T_k^0$ densely on $\R^+$ for all $i,j,k\in\{1,\ldots,p\}$.
\end{lemma}
\begin{proof} If $p=1$ this is the second statement of Lemma \ref{stably equiv 1}. If $p>1$, let $i\ne j$ be so that $T_i\sim_s T_j$ densely on $\R^+$ (from the first statement of Lemma \ref{stably equiv 1}). From Corollary \ref{cor} we know that $T_i$ and $T_j$ are not asymptotic on $\R^+$. Thus there are patches $\{\tau,\tau'\}\subset T_i$ and $\{\tau,\tau''\}\subset T_j$ with support contained in $\R^+$, $\tau'\ne\tau''$, and $min(\tau')=min(\tau'')=max(\tau)$.
From Property 2, one of $\tau',\tau''$ is of type 1: say $\tau'=\tau_1+t$ and $\tau''=\tau_r-min(\tau_r)+t$, with $r\in\{2,\ldots,m+p\}$. Let $l\in\{m+1,\ldots,m+p\}$ be in phase with $r$ (that is, $p|(l-r)$). Then $\Pb^{qp}(\tau'')=\{\tau_l-min(\tau_l)+\beta^{qp} t,\ldots\}$ for some $q$. Let $\epsilon=min\{length(\tau_1),length(\tau_r),\beta^{-qp}\cdot length(\tau_l)\}$. Then $\tau_1\sim_s\tau_r-min(\tau_r)\sim_s\tau_l-min(\tau_l)$ densely on $[0,\epsilon)$. It follows that $T_l\sim_s T_k^0$ densely on $\R^+$ for all $k\in\{1,\ldots,p\}$. Thus, with subscripts taken mod$(p)$, $T_{l+n}=\Pb^n(T_l)\sim_s \Pb^n(T_k^0)=T_{k+n}^0$ densely on $\R^+$ for all $n\in\N$.
\end{proof}

The {\em stable equivalence relation}, $\approx_s$, is defined for a substitution tiling space $\Omega_{\phi}$ by $T\approx_sT'$ if and only if $T\sim_sT'$ densely on $\R$ and $T\sim_{srp}T'$. 

\begin{cor} \label{approx nontrivial} If $m>0$ then $T_j\approx_s T_j^0$ for each $j\in\{1,\ldots,p\}$.
\end{cor}

The following is Theorem 5 of \cite{B3}.

\begin{Theorem} \label{pds means stab equiv} If $\phi$ is a primitive, non-periodic Pisot substitution, then $(\Ob,\R)$ has pure discrete spectrum if and only if stable equivalence equals strong regional proximality on $\Ob$. 
\end{Theorem}

Given a morphism $\gamma$ from letters of an alphabet $\mathcal{A}$ to nonempty words on an alphabet $\mathcal{A}'$, let $i_{\gamma},t_{\gamma}:\mathcal{A}\to\mathcal{A}'$ be the initial and terminal letter maps: If $\gamma(i)=a\cdots z$ then $i_{\gamma}(i)=a$ and $t_{\gamma}(i)=z$.

The stable equivalence relation is pushed down to a relation $R$ on the rose $\mathcal{R}_{\phi}$ by $xRy$ if and only if there are $T_1,\ldots,T_n, T_1',\ldots,T_n'\in\Omega_{\phi}$ with $T_i\approx_sT_i'$ for $i=1,\ldots, n$, $g_{\phi}(T_{i+1})=g_{\phi}(T_i')$ for $i=1,\ldots, n-1$, $g_{\phi}(T_1)=x$, and $g_{\phi}(T_n')=y$. The following lemma is proved in \cite{B1} under the slightly stronger assumption that either $t_{\phi}$ or $i_{\phi}$ is eventually constant (see Lemmas3.5 and 3.6 there), but the proof goes through virtually without change under the condition we assume here of
 dense terminal or dense initial stable relation.
 
\begin{lemma} \label{R} Suppose that $\phi$ is a nonperiodic Pisot substitution with the properties: $(\Omega_{\phi},\R)$ does not have pure discrete spectrum and there is $\epsilon>0$ so that either $\tau-min(\tau)\sim_s \tau'-min(\tau')$ densely on $[0,\epsilon)$ for all tiles $\tau,\tau'$, or $\tau-max(\tau)\sim_s \tau'-max(\tau')$ densely on $(-\epsilon,0]$ for all tiles $\tau,\tau'$. Then $R$ is a closed equivalence relation on $\mathcal{R}_{\phi}$ with boundedly finite equivalence classes.
\end{lemma}

Recall (definition \ref{def of f}) that $f_{\phi}:\mathcal{R}_{\phi}\to\mathcal{R}_{\phi}$ is the map induced by $\Phi:\Omega_{\phi}\to\Omega_{\phi}$ (so that $f_{\phi}\circ g_{\phi}=g_{\phi}\circ\Phi$). Then $xRy\implies f_{\phi}(x)Rf_{\phi}(y)$ and there is an induced map $\tilde{f}_{\phi}:\mathcal{R}_{\phi}/R\to\mathcal{R}_{\phi}/R$. Under the assumptions of Lemma \ref{R}, $\mathcal{R}_{\phi}/R$ is a wedge of circles, $\tilde{f}_{\phi}$ defines a substitution $\tilde{\phi}$, and we may identify $\mathcal{R}_{\phi}/R$ with $\mathcal{R}_{\tilde{\phi}}$. If $\mathcal{A}$ and $\tilde{\mathcal{A}}$ are the alphabets for $\phi$ and $\tilde{\phi}$ then the quotient map from $\mathcal{R}_{\phi}$ to $\mathcal{R}_{\tilde{\phi}}$ induces a morphism $\alpha:\mathcal{A}\to\tilde{\mathcal{A}}^*$ so that $\alpha\circ\phi=\tilde{\phi}\circ\alpha$. 

The following is distilled from \cite{B1} (see, in particular, Theorem 3.9 there).

\begin{Theorem} \label{reduction} Suppose that $\phi$ is a nonperiodic Pisot substitution with the properties: $(\Omega_{\phi},\R)$ does not have pure discrete spectrum and there is $\epsilon>0$ so that either $\tau-min(\tau)\sim_s \tau'-min(\tau')$ densely on $[0,\epsilon)$ for all tiles $\tau,\tau'$; or $\tau-max(\tau)\sim_s \tau'-max(\tau')$ densely on $(-\epsilon,0]$ for all tiles $\tau,\tau'$. Then $\tilde{\phi}$ is a nonperiodic Pisot substitution and 
$(\Omega_{\tilde{\phi}},\R)$ does not have pure discrete spectrum.
\end{Theorem}

To apply Theorem \ref{reduction} to $\beta$-substitutions we will require the following lemma.

\begin{lemma} \label{initial stability} If $m>0$, there is $\epsilon>0$ so that $\tau-min(\tau)\sim_s\tau'-min(\tau')$ densely on $[0,\epsilon)$ for all tiles $\tau,\tau'$  for $\pb$.
\end{lemma}
\begin{proof}
For each $i>1$ there is $j\in\{1,\ldots,p\}$ so that $z^i$ is in phase with $z^{m+j}$. This means that $\Pb^{m-i+1}(\tau_+^i)=\{\tau_+^{m+1},\ldots\}$ and $\Pb^{m-i+1}(\tau_+^{m+j})=\{\tau_+^{m+1},\ldots\}$. Hence $\tau^i\sim_s\tau^{m+j}$ densely on $[0,\epsilon_i)$, with $\epsilon_i:=\beta^{i-m-1}(length(\tau^{m+1}))$. Since $T_i\sim_sT_j\sim_sT_k^0$ densely on $\R^+$ (Lemma \ref{stably equiv 2}) we have that $\tau_1\sim_s\tau_+^l$ densely on $[0,\epsilon)$, with $\epsilon:=min\{\beta^{-m-1}(length(\tau_+^{m+1})),\,length(\tau_1)\}$, for all $l\in\{2,\ldots,m+p\}$.
\end{proof}

\begin{lemma} \label{eventually} Suppose that $(\Ob,\R)$ does not have pure discrete spectrum and $m>0$. Then $i_{\tpb}$ is eventually constant and $t_{\tpb}$ is injective.
\end{lemma}

\begin{proof}
Pick $i\in\{1,\ldots,p\}$ and let $P\subset T_i$ and $P'\subset T_i^0$ be patches, both supported on $[0,t_0]$, with $P\cap P'=\emptyset$ and so that $\tau_1+t_0\in B_0[T_i-t_0]\cap B_0[T_i^0-t_0]$ (that is, $P,P'$ are the initial patches of $T_i,T_i^0$ preceding their first agreement). Then $\Pb^p(P)=P\cup \{\tau_1+t_0\}\cup Q$ and $\Pb^p(P')=P'\cup\{\tau_1+t_0\}\cup Q'$.
Let $k>0$ be minimal with the property that $\pb^k(P)\cap\Pb^k(P')\ne\emptyset$. We may as well assume (after replacing $T_i$ and $T_i^0$ by $\Pb^{k-1}(T_i)$ and $\Pb^{k-1}(T_i^0)$, and $P,P'$ by the initial disjoint patches  of $\Pb^{k-1}(T_i),\Pb^{k-1}(T_i^0)$) that $k=1$. Then $\Pb(P)=P_1\cup \{\tau_1+t\}\cup P_2$ and $\Pb(P')=P_1'\cup \{\tau_1+t\}\cup P_2'$ for some $t>0$, with $P_1\cap P_1'=\emptyset$ and $spt(P_1)=[0,t]=spt(P_1')$. Consider the (disjoint) patches $B_0[T_i-t/\beta]\subset P$ and $B_0[T_i^0-t/\beta]\subset P'$. Let us first observe that these patches can't both be doubletons. Indeed, if $B_0[T_i-t/\beta]=\{\rho_1<\rho_2\}$ and $B_0[T_i^0-t/\beta]=\{\rho_1'<\rho_2'\}$, then $\rho_2=\tau_1+t/\beta=\rho_2'$ (since $\tau_1$ is the only prototile whose image under $\Pb$ begins with a tile of type 1), contradicting the disjointness of $P$ and $P'$. If $B_0[T_i-t/\beta]$ and $B_0[T_i^0-t/\beta]$ are both singletons, then these must be of the form $\{\tau\}$, $\{\tau'\}$ with $\tau$ a translate of some $\tau_l$ and $\tau'$ a translate of some $\tau_{l'}$ with $l,l'$ such that $k/\beta\in int(spt(\tau_l))$ and $k'/\beta\in int(spt(\tau_{l'}))$ for some $k,k'\in\{1,\ldots \lfloor\beta\rfloor\}$. But then $\tau_{m+p}+t-1\in \Pb(\tau)\cap\Pb(\tau')$ is in both $P_1$ and $P_1'$, contradicting the disjointness of these patches. Thus, exactly one of $B_0[T_i-t/\beta]$ and $B_0[T_i^0-t/\beta]$ is a doubleton and one of the patches $P_1, P_1'$ ends in $\tau_{m+p}+t-1$ and the other ends in $\tau_k+t-max(\tau_k)$ for some $k\in \{1,\ldots,m+p-1\}$.

Since $\Pb(T_i)\approx_s\Pb(T_i^0)$, we see from the definition of the relation $R$ that $q(1-s)Rq(max(\tau_k)-s)$ in the rose $\mathcal{R}_{\beta}$ for all positive $s$ less than the smaller of the lengths of $\tau_k$ and $\tau_{m+p}$. 
This means that $t_{\alpha}(k)=t_{\alpha}(m+p)$. Let $max({\tau_k})=z^{n+1}$. Then $t_{\pb^n}(m+p)=k$ and it follows from $\alpha\circ\pb=\tpb\circ\alpha$ that $t_{\tpb^n}(t_{\alpha}(m+p))=t_{\alpha}(t_{\pb^n}(m+p))=t_{\alpha}(k)=t_{\alpha}(m+p)$. Now let $a$ be any letter in the alphabet for $\tpb$ and let $l\in\{1,\ldots,m+p\}$ be so that $t_{\alpha}(l)=a$. Let $r\in\N$ be so that $t_{\pb^r}(m+p)=l$. Then

$$t_{\tpb^n}(a)=t_{\tpb^n}(t_{\alpha}(l))$$
$$=t_{\tpb^n}(t_{\alpha}(t_{\pb^r}(m+p)))$$
$$=t_{\tpb^r}(t_{\alpha}(t_{\pb^n}(m+p)))$$
$$=t_{\tpb^r}(t_{\alpha}(m+p))$$
$$=t_{\alpha}(t_{\pb^r}(m+p))$$
$$=t_{\alpha}(l)$$
$$=a.$$
Hence $t_{\tpb}$ is injective.

Pick $j\in\{1,\ldots,p\}$ (so $T_j\approx_s T_j^0$ by Corollary \ref{approx nontrivial}) and let $k\in\{2,\ldots,m+p\}$ be so that $min(\tau_k)=z^{m+j}$ (thus $\tau_k$ is the initial tile of $T_j$). Then $i_{\alpha}(k)=i_{\alpha}(1)$. Now, for any $l\in\{2,\ldots, m+p\}$ there is $n\in\N$ so that $i_{\pb^n}(l)=k$, and then $i_{\tpb^n}(i_{\alpha}(l))=i_{\alpha}(i_{\pb^n}(l))=i_{\alpha}(k)=i_{\alpha}(1)$. Since $i_{\pb}(1)=1$, $i_{\tpb}(i_{\alpha}(1))=i_{\alpha}(1)$. That $i_{\tpb}$ is eventually constant (with value $i_{\alpha}(1)$) follows from surjectivity of $i_{\alpha}$.

\end{proof}

The following is Theorem 3.12 of \cite{B1}.

\begin{Theorem} \label{i and t} Suppose that  $\phi$ is a nonperiodic Pisot substitution with the properties: $i_{\phi}$ is eventually constant and $t_{\phi}$ is injective. Then $(\Omega_{\phi},\R)$ has pure discrete spectrum.
\end{Theorem}

\begin{theorem} \label{main theorem} If $\beta$ is Pisot, then $(\Ob,\R)$ has pure discrete spectrum.
\end{theorem}
\begin{proof}This is clearly true if $\beta=n\in\N$. Suppose that $\beta$ is Pisot, and not an integer. If $m=0$ it is easy to check that $i_{\pb}$ is eventually constant and $t_{\pb}$ is injective. Hence $(\Ob,\R)$ has pure discrete spectrum by Theorem \ref{i and t}. If $m>0$ and $(\Ob,\R)$ does not have pure discrete spectrum, then the substitution $\tpb$ is non-periodic Pisot and $(\Omega_{\tpb},\R)$ does not have pure discrete spectrum (Lemma \ref{initial stability} to get $\tilde{\pb}$ and Theorem \ref{reduction}). But $(\Omega_{\tpb},\R)$ must have pure discrete spectrum by Lemma \ref{eventually} and Theorem \ref{i and t}.
\end{proof}

\begin{remark} If the degree $d$ of the Pisot number $\beta$ equals $m+p$ (so that the $\beta$-substitution is irreducible), then it follows from Theorem \ref{main theorem} and \cite{CS} (see, also, \cite{Sa}) that the substitutive system (not to be confused with the $\beta$-shift) also has pure discrete spectrum. There are, however, examples of Pisot $\beta$ for which the substitutive system does not have pure discrete spectrum (\cite{EI}. It follows from Theorem \ref{main theorem} and Proposition 8.1 of \cite{BBK} that, for Pisot units, the substitutive system associated with $\pb$ has, as an a.e. one-to-one factor, an isometric exchange of domains (the Rauzy pieces) induced as a first return of a minimal translation on the $(d-1)$-torus.
\end{remark}

\begin{corollary} \label{weak} All Pisot numbers are weakly finitary.
\end{corollary}

\begin{corollary} \label{coding} If $\beta$ is Pisot, $M$ is the companion matrix for the minimal polynomial of $\beta$, and $\bar{y}$ is a fundamental homoclinic point for the solenoidal automorphism $\hat{F}_M$, then the arithmetical coding $h_{\bar{y}}:X_{\beta}\to\hat{\T}^d_{\beta}$ is $a.e.$ one-to-one.
\end{corollary}

We prove Corollaries \ref{weak} and \ref{coding} in the next section.

\section{ The connection between pure discrete spectrum, arithmetical coding, and Property (W).} \label{connections}

For this section, fix a Pisot number $\beta$ of algebraic degree $d$. The {\em substitution matrix} $A$ associated with $\pb$ is the $(m+p)\times(m+p)$ matrix with $ij$-th entry $a_{ij}$ given by the number of occurances of the letter $j$ in the word $\pb(i)$. Then $\beta$ is a simple eigenvalue of $A$ and the characteristic polynomial of $A$ factors over $\Z$ as 
$$p_A(x)=p_{\beta}(x)q(x)$$
with $p_{\beta}(x)$ the minimal polynomial of $\beta$ and $q(x)$ relatively prime with $p_{\beta}(x)$. There are then $s_1(x),s_2(x)\in\Q[x]$ with $s_1(x)p_{\beta}(x)+s_2(x)q(x)=1$  and an $A$-invariant splitting
$$\R^{m+p}=V\oplus W,$$
where $V:=ker(s_1(A)p_{\beta}(A))$ and $W:=ker(s_2(A)q(A))$ are rational (that is, spanned by rational vectors).
The {\em Pisot subspace} $V$ is $d$-dimensional and the linear transformation $L:=A|_V$ has characteristic polynomial $p_{\beta}(x)$.

 Let $V=E^s\oplus E^u$
be the splitting of $V$ into $(d-1)$-dimensional stable and 1-dimensional unstable spaces, let $\pi_{V}:\R^{m+p}\to V$ be projection along $W$, and, for $v\in V$, write $v=v^s+v^u$ with $v^{s,u}\in E^{s,u}$. By $\pi^s$ (resp., $\pi^u$) we will denote both projection of $\R^{m+p}$ and $V$ along $W\oplus E^u$ (resp., $W\oplus E^s$) and of $V$ along $E^u$ (resp., $E^s$) onto $E^s$ (resp., $E^u$).
 The subgroup $$\Gamma:=\pi_V(\Z^{m+p})$$ of $V$ is a $d$-dimensional lattice.
 
 It is convenient to view tilings $T\in\Ob$ as maps of $\R$ into the rose $\mathcal{R}_{\pb}$, rather than as collections of tiles: Given $T\in\Ob$, let $T(t):=g_{\pb}(T-t)$ (recall the definition \ref{def of g}). The $\R$-action on tilings-as-maps is then given by $(T-t)(s):=T(s+t)$; $\Pb$ translates as $\Pb(T)(t):=f_{\pb}(T(t/\beta))$ (with $f_{\pb}$ as in \ref{def of f}); and the tiling metric induces the topology of unifom convergence on compact sets. We view the prototiles as parameterizations of the corresponding petals of the rose via
 $\tau_i(t):=T(t-min(\tau_i))$ where $\tau_i\in T\in\Ob$.
 
For each $i\in\{1,\ldots,m+p\}$, let $e_i\in\Z^{m+p}$ denote the standard unit vector and let $\sigma_i:=\{te_i:t\in[0,1]\}\subset\R^{m+p}$. Let $\tilde{\mathcal{R}}_{\pb}\subset \R^{m+p}$ be the grid
$$\tilde{\mathcal{R}}_{\pb}:=\cup_{a\in\Z^{m+p}}\cup_{i=1,\ldots,m+p}(a+\sigma_i).$$
The map $\tilde{\pi}:\tilde{\mathcal{R}}_{\pb}\to\mathcal{R}_{\pb}$ given by $$\tilde{\pi}(a+te_i):=\tau_i(t)$$
is the {\em universal abelian cover} of $\mathcal{R}_{\pb}$.

We denote by $\tilde{f}_{\pb}:\tilde{\mathcal{R}}_{\pb}\to\tilde{\mathcal{R}}_{\pb}$ the unique lift of $f_{\pb}$ with $\tilde{f}_{\pb}(0)=0$ and, if $T\in\Ob$, $\tilde{T}:\R\to\tilde{\mathcal{R}}_{\pb}$ will denote any lift of $T:\R\to \mathcal{R}_{\pb}$. We then define $\tilde{\Psi}_{\beta}$ on such lifts by $$\tilde{\Psi}_{\beta}(\tilde{T})(t):=\tilde{f}_{\pb}\circ\tilde{T}(t/\beta).$$
Note that since $A$ is the `abelianization' of $f_{\pb}$, $\tilde{f}_{\pb}(a)=Aa$ for $a\in\Z^{m+p}$.

By a {\em strand $\gamma$ corresponding to $T\in\Ob$} we will mean any $$\gamma=x+\pi_V\circ\tilde{T},$$
with $x\in V$ and $\tilde{T}$ a lift of $T$. For $T\in\Ob$, let 

\begin{equation} \label{def of t_*}
t_*(T):=sup\{t\le0:\tau_1+t\in T\}.
\end{equation}
So $T(t_*(T))=*$ is the branch point of $\mathcal{R}_{\pb}$ and $t_*(T)$ is the largest non-positive occurance of the initial vertex of a tile of type 1 in $T$. For a strand $\gamma$ corresponding to $T$ we define its {\em initial vertex} to be 

\begin{equation} \label{def of a}
a(\gamma):=\gamma(t_*(T))
\end{equation}
and the  (abelian) {\em prefix} of $T$, $p(T)$, to be

\begin{equation} \label{def of p}
p(T):=\gamma(t_*(T))-\gamma(\beta t_*(\Pb^{-1}(T))).
\end{equation}

Let us extend $\Pb$ to strands by setting $\Pb(x+\pi_V\circ \tilde{T}):=Lx+\pi_V\circ\tilde{\Psi}_{\beta}(\tilde{T})$. Then, if $\gamma$ is a strand corresponding to $T$,  $\Pb(\gamma)$ is the unique strand corresponding to $\Pb(T)$ having initial vertex $a(\gamma')=p(\Pb(T))+L\gamma (t_*(T))$.
It is clear that $\Pb$ is invertible on strands (even though, if $det(L)\ne\pm1$, $\tilde{\Psi}_{\beta}$ is not invertible), and
$$p(T)=a(\gamma)-La(\Pb^{-1}(\gamma))$$ for any strand $\gamma$ corresponding to $T$. We extend the $\R$-action to lifts and strands by $(\tilde{T}-t)(s):=\tilde{T}(s+t)$ and $(\gamma-t)(s):=\gamma(s+t)$. One easily checks that $\tilde{\Psi}_{\beta}(\tilde{T}-t)=\tilde{\Psi}_{\beta}(\tilde{T})-\beta t$ and $\Pb(\gamma-t)=\Pb(\gamma)-\beta t$.

Given a strand $\gamma$ corresponding to $T\in\Ob$, consider the sequence $(a_i(\gamma))_{i\in\Z}$  of initial vertices of the strands in the $\Pb$-orbit of $\gamma$:
$$a_i(\gamma):=\Pb^i(\gamma)(t_*(\Pb^i(T))).$$
We see that $|a_i(\gamma)-La_{i-1}(\gamma)|=|p(\Pb^i(T))|$ is bounded. From the hyperbolicity of $L$ it follows that $(a_i(\gamma))$ is {\em globally shadowed} by a unique point in V:

\begin{lemma} \label{gs} Given a strand $\gamma$, there is a unique point $G(\gamma)\in V$ with the property that $|L^iG(\gamma)-a_i(\gamma)|$, $i\in\Z$, is bounded. Furthermore, $a(\gamma)-G(\gamma)$ is uniformly bounded as a function of $\gamma$.
\end{lemma}
\begin{proof}
We have $L^{-i}a_i^u(\gamma)=a_0^u(\gamma)+L^{-1}p^u(\Pb^0(T))+\cdots+L^{-i}p^u(\Pb^{i-1}(T))$ and 
$L^ia_{-i}^s(\gamma)=a_0^s(\gamma)-p^s(\Pb^{-1}(T))-\cdots-L^{i-1}p^s(\Pb^{-i}(T))$. Since the $p(\Pb^i(T))$ are bounded, the limits in the following formula for $G(\gamma)$ exist:
$G(\gamma)=\lim_{i\to\infty}L^{-i}a_i^u(\gamma)+\lim_{i\to\infty}L^ia^s_{-i}(\gamma)$.
\end{proof}

Note that if $\gamma'=x+\gamma$ is another strand corresponding to $T$, then $a_i(\gamma')=L^ix+a_i(\gamma)$, so $G(\gamma')=G(\gamma)+x$. Thus $$\gamma_T:=\gamma-G(\gamma)$$
is the unique strand corresponding to $T$ with the property that $(a_i(\gamma_T))$ is bounded. We define $\pi:\Ob\to V/\Gamma$  by 

\begin{equation}\label{def of pi}
\pi(T):=a_0(\gamma_T)+\Gamma.
\end{equation}

Let $l=(l_1,\ldots,l_{m+p})$ with $l_i:=length(\tau_i)$ the lenngth of the $i$-th prototile. Then $l$ is a left eigenvector of $A$ (for the eigenvalue $\beta$) and is orthogonal to $E^s$. Thus, setting $\omega$ to be the right positive eigenvector of $A$, normalized with $\langle l^t,\omega\rangle=1$, we have $v^u=\pi^u(v)=\langle l^t,v\rangle\omega$ for all $v\in V$. It follows that if $\gamma$ is any strand, then 

\begin{equation}\label{gamma^u}
\gamma^u(t)-\gamma^u(t')=(t-t')\omega
\end{equation}
for all $t,t'\in\R$.

\begin{lemma} \label{preserves action} $\pi(T-t)=\pi(T)-(t\omega+\Gamma)$ for all $T\in\Ob$ and $t\in\R$.
\end{lemma}
\begin{proof}
Since $\Pb^i(\gamma-t)=\Pb^i(\gamma)-\beta^it$ for any strand $\gamma$, $i\in\Z$ and $t\in\R$, $|a_i(\gamma-t)-a_i(\gamma)|$ is bounded (for fixed $\gamma$, $t$) for $i\le0$. Now $|a_i(\gamma)-\Psi_{\beta}^i(\gamma)(0)|$ and $|a_i(\gamma-t)-\Psi_{\beta}^i(\gamma-t)(0)|$ are bounded and $(\Pb^i(\gamma)-\beta^it)^u(0)-(\Pb^i(\gamma))^u(0)=\beta^it\omega$. Thus $|a_i^u(\gamma-t)-a_i^u(\gamma)-\beta^it\omega|$ is bounded for $i\ge 0$. It follows that $G(\gamma-t)=G(\gamma)-t\omega$.
Note that $a_0(\gamma-t)\equiv a_0(\gamma)\mod(\Gamma)$. We have $\pi(T-t)=a_0(\gamma_{(T-t)})+\Gamma=a_0((\gamma_T-t)-G(\gamma_T-t))+\Gamma=a_0(\gamma_T)-G(\gamma_T)-t\omega+\Gamma=\pi(T)-(t\omega+\Gamma)$ (using $G(\gamma_T)=0$).
\end{proof}

\begin{lemma} \label{factor} $\pi:\Ob\to V/\Gamma$ is a continuous surjection.
\end{lemma}
\begin{proof} Continuity is clear from the explicit formula for $G$ given in the proof of Lemma \ref{gs}. Surjectivity is a consequence of Lemma \ref{preserves action} and the irrationality of $\omega$: If $cl\{t\omega:t\in\R\}$ were a proper sub-torus $\T$ of dimension $ k<d$, then there would be a rational $k$-dimensional subspace $U$ of $V$, invariant under $L$. But then the degree of $\beta$ would be at most $k$, and not $d$.
\end{proof} 

The previous two lemmas combine to say that $\pi$ is factor map of $(\Ob,\R)$ onto $(\T^d,\R)$, with the $\R$-action on $\T^d\simeq V/\Gamma$ given by $(x+\Gamma)-t:=x-t\omega +\Gamma$. It is easy to check that 
$\pi$ also semiconjugates $\Pb$ with $F_L:V/\Gamma\to V/\Gamma$ by $F_L(x+\Gamma)=Lx+\Gamma$. Thus $\pi$ induces 

\begin{equation}\label{pihat}
\hat{\pi}:\Ob\to \hat{\T}^d:=\inv F_L,
\end{equation}
which factors $(\Ob,\R)$ onto $(\hat{\T}^d,\R)$.

For $T,T'\in\Ob$, we say that {\em $T$ globally shadows $T'$}, and write $T\sim_{gs}T'$, provided there are lifts $\tilde{T},\tilde{T}'$ of $T,T'$ so that the strands $\gamma:=\pi_V(\tilde{T}),\gamma':=\pi_V(\tilde{T}')$ satisfy: $|a_i(\gamma)-a_i(\gamma')|$ is bounded.

If $\beta$ is a unit, then $det(L)=\pm 1$, $\tilde{\Psi}_{\beta}$ is invertible on $\{\tilde{T}:T\in\Ob\}$, and the above definition can be rephrased as: $T\sim_{gs}T'$ provided there are lifts $\tilde{T},\tilde{T}'$ so that $|\pi_V\circ\tilde{\Psi}_{\beta}^i(\tilde{T})(t_{*,i})-\pi_V\circ\tilde{\Psi}_{\beta}^i(\tilde{T}')(t'_{*,i})|$ is bounded, where $t_{*,i}:=t_*(\Pb^i(T))$ and $t'_{*,i}:=t_*(\Pb^i(T'))$. To cover the `nonunimodular' case as well, we will need the following awkward (but straightforward) reformulation.

\begin{lemma} \label{reformulation} $T\sim_{gs}T'$ if and only if there are $B<\infty$ and lifts $\tilde{T}$, $\tilde{T}'$, so that for each $i\in\N$ there are $v_i\in\Gamma$ and lifts $\tilde{T}_i,\tilde{T}'_i$ of $\Pb^{-i}(T),\Pb^{-i}(T')$, such that 
$\tilde{\Psi}_{\beta}^i(\tilde{T}_i)=v_i+\tilde{T}$, $\tilde{\Psi}_{\beta}^i(\tilde{T}'_i)=v_i+\tilde{T}'$, and $|\pi_V\circ\tilde{\Psi}_{\beta}^k(\tilde{T}_i)(t_{*,k-i})-\pi_V\circ\tilde{\Psi}_{\beta}^k(\tilde{T}'_i)(t'_{*,k-i})|\le B$ for all $k\in\N$, where $t_{*,j}:=t_*(\Pb^j(T))$ and $t'_{*,j}:=t_*(\Pb^j(T'))$.
\end{lemma}

\begin{lemma} $\hat{\pi}(T)=\hat{\pi}(T')$ if and only if $T\sim_{gs} T'$.
\end{lemma}
\begin{proof}
Suppose that $\hat{\pi}(T)=\hat{\pi}(T')$ and fix $i\ge0$. Let $\gamma_{\Pb^{-i}(T)}=x+\pi_V\circ\tilde{T}_i$ and $\gamma_{\Pb^{-i}(T')}=x'+\pi_V\circ \tilde{S}_i$ with $\tilde{T}_i$ and $\tilde{S}_i$ lifts of $\Pb^{-i}(T)$ and $\Pb^{-i}(T')$, resp. Then $\pi(\Pb^{-i}(T))=\pi(\Pb^{-i}(T'))$ means that $x-x'\in\Gamma$. There is $B$ so that $|a(\gamma_T)|<B/2$ for all $T$ (Lemma \ref{gs}), so $|(a_k(\gamma_{\Pb^{-i}(T)}) -x)- (a_k(\gamma_{\Pb^{-i}(T')})-x)|=|(a_k(\gamma_{\Pb^{-i}(T)} -x))- (a_k(\gamma_{\Pb^{-i}(T')}-x))|<B$ for all $i,k\ge0$. Let $y\in\Z^{m+p}$ be such that $\pi_V(y)=x-x'$ and set $\tilde{T}_i':= y+\tilde{S}_i$. Then, with $t_{*,j}:=t_*(\Pb^j(T))$ and $t_{*,j}':=t_*(\Pb^j(T'))$, we have  $|\pi_V\circ\tilde{\Psi}_{\beta}^k(\tilde{T}_i)(t_{*,k-i})-\pi_V\circ\tilde{\Psi}_{\beta}^k(\tilde{T}'_i)(t'_{*,k-i})|=|(a_k(\gamma_{\Pb^{-i}(T)} -x))- (a_k(\gamma_{\Pb^{-i}(T')}-x))|<B$ for $k\in\N$. Letting $\tilde{T}=\tilde{T}_0$ and $\tilde{T}'=\tilde{T}_0'$ we have $\tilde{\Psi}_{\beta}^i(\tilde{T}_i)=\tilde{T}+v_i$ and  $\tilde{\Psi}_{\beta}^i(\tilde{T}'_i)=\tilde{T}'+v_i$ for some $v_i\in\Gamma$; hence $T\sim_{gs}T'$.

Now suppose that $T\sim_{gs}T'$ and let $B>0$, $v_i$, $\tilde{T}$, $\tilde{T}'$, $\tilde{T}_i$ and $\tilde{T}'_i$ be as in the reformulation of $\sim_{gs}$ (Lemma \ref{reformulation}).
Let $\gamma:=\pi_V\circ\tilde{T}$ and $\gamma':=\pi_V\circ\tilde{T}'$. Then $\Pb^{-i}(\gamma)=x_i+\pi_V\circ\tilde{T}_i$ and $\Pb^{-i}(\gamma')=x_i+\pi_V\circ\tilde{T}'_i$ with $x_i=-1/\beta^{i}\pi_V(v_i)$. Then $a(\gamma)\equiv a(\gamma') \mod(\Gamma)$ and from $|\pi_V\circ\tilde{\Psi}_{\beta}^k(\tilde{T}_i)(t_{*,k-i})-\pi_V\circ\tilde{\Psi}_{\beta}^k(\tilde{T}'_i)(t'_{*,k-i})|\le B$ for all $k\in\N$, we see that $G(\gamma)=G(\gamma')$. Thus, $a(\gamma_T)\equiv a(\gamma)-G(\gamma)\equiv a(\gamma')-G(\gamma')\equiv a(\gamma_{T'})\mod(\Gamma)$, so $\pi(T)=\pi(T')$. It is clear from the definition of $\sim_{gs}$ that $T\sim_{gs}T'\implies \Pb^{k}(T)\sim_{gs}\Pb^{k}(T')$ for all $k\in\Z$. Thus $\pi(\Pb^{-k}(T))=\pi(\Pb^{-k}(T'))$ for all $k\ge0$ and $\hat{\pi}(T)=\hat{\pi}(T')$. 
\end{proof}

Translated into the context of tilings-as-maps, the strong regional proximal relation reads as follows: $T\sim_{srp} T'$ provided for each $R>0$ there are $S_R,S'_R\in\Ob$ and $t_R\in\R$ so that $T(t)=S_R(t)$, $T'(t)=S'_R(t)$, and $S_R(t-t_R)=S'_R(t-t_R)$ for all $|t|\le R$.

\begin{theorem} \label{rp=gs}If $\beta$ is  Pisot then the strong regional proximal and global shadowing relations on $\Ob$ are the same.
\end{theorem}

This can be deduced (at least, in the unimodular case) as a corollary to Propositions 24 and 25 of \cite{BG} which combine for a similar result in the higher dimensional setting. We give a simplified argument for the one-dimensional case at hand.

Let's say that strands $\gamma,\gamma'$ corresponding to $T,T'\in\Ob$ are strong regionally proximal, and write $\gamma\sim_{srp}\gamma'$, provided there are: lifts $\tilde{T},\tilde{T}'$ of $T,T'$ and $x\in V$ with $\gamma=x+\pi_V(\tilde{T}),\gamma'=x+\pi_V(\tilde{T}')$; and, for each $R>0$, there are $t_R\in\R$ and $S_R,S'_R\in\Ob$ with lifts $\tilde{S}_R, \tilde{S}'_R$ so that $\tilde{T}(t)=\tilde{S}_R(t),\tilde{T}'(t)=\tilde{S}'_R(t), \tilde{S}_R(t-t_R)=\tilde{S}'_R(t-t_R)$, for all $|t|\le R$.

\begin{lemma} \label{B} There is $B\in\R$ so that if $\gamma\sim_{srp}\gamma'$, then $|\gamma(0)-\gamma'(0)|\le B$.
\end{lemma}
\begin{proof}
Let $n\in\N$ and $\rho<1$ be so that $|(Lx)^s|\le\rho$ for all $x\in X$ with $|x^s|\le 1$. Let $B_1$ be such that if $j:[a,b]\to\mathcal{R}_{\pb}$ is any parameterization of a petal of $\mathcal{R}_{\pb}$ with lift $\tilde{j}$, then $|\pi^s(\pi_V(\tilde{f}_{\pb}^n(\tilde{j}(t_1))))-\pi^s(\pi_V(\tilde{f}_{\pb}^n(\tilde{j}(t_2))))|\le B_1$ for all $t_1,t_2\in[a,b]$. Now take $B$ large enough so that $B/2-\rho B/2\ge B_1$. Then if $\gamma$ is any strand corresponding to an element of $\Ob$, $|\gamma^s(t_1)-\gamma^s(t_2)|\le B/2$ for all $t_1,t_2\in\R$ (since any compact piece of the image of $\gamma$ is contained in a translation of some $\tilde{f}^{kn}(\tilde{j}([a,b])$). Now if $\gamma$ and $\gamma'$ are strong regionally proximal, with $S_1,S'_1$ as in the above definition and $R=1$, then $|(\gamma(0))^s-(\pi_V(\tilde{S}(t_1)))^s|\le\frac{B}{2}$ and $|(\gamma'(0))^s-(\pi_V(\tilde{S}'(t_1)))^s|\le\frac{B}{2}$. Since $\tilde{S}(t_1)=\tilde{S}'(t_1)$, we have $(\gamma(0))^u=(\gamma(0))^u$ and $|(\gamma(0))^s-(\gamma'(0))^s|\le B$.

\end{proof}

\begin{lemma} \label{lemrp} If $\gamma\sim_{srp}\gamma'$ then $\Pb(\gamma)\sim_{srp}\Pb(\gamma')$ and $\Pb^{-1}(\gamma)\sim_{srp}\Pb^{-1}(\gamma')$.
\end{lemma}

\begin{proof} Suppose that $\gamma$ and $\gamma'$ correspond to $T,T'\in\Ob$ and $\gamma\sim_{srp}\gamma'$. Let $R>0$ be given and let $S_{\beta R},S'_{\beta R}\in\Ob$, $t_{\beta R}\in\R$, be as in the above definition of $\sim_{srp}$, with $\beta R$ replacing $R$. We may then lift $\Pb^{-1}(T),\Pb^{-1}(S_{\beta R}), \Pb^{-1}(S'_{\beta R})$, and $\Pb^{-1}(T')$, successively, so that $\widetilde{\Pb^{-1}(T)}(t)=\widetilde{\Pb^{-1}(S_{\beta R})}(t)$, $\widetilde{\Pb^{-1}(S_{\beta R})}(t-(1/\beta)t_{\beta R})=\widetilde{\Pb^{-1}(S_{\beta R})}(t-(1/\beta)t_{\beta R})$, and $\widetilde{\Pb^{-1}(T')}(t)=\widetilde{\Pb^{-1}(S'_{\beta R})}(t)$ for all $|t|\le R$. Projecting to strands, and appropriately translating, shows that $\Pb^{-1}(\gamma)\sim_{srp}\Pb^{-1}(\gamma')$. The argument that $\Pb(\gamma)\sim_{srp}\Pb(\gamma')$ is similar.

\end{proof}

Recall that $\sigma_i:=\{te_i:0\le t\le 1\}\subset \R^{m+p}$, where $e_i$ is the $i$-th standard unit vector. By {\em the segment of type $i$ at $x\in\Gamma$} we will mean $(x+\pi_V(\sigma_i),i)$.
Let $\Sigma:=\{(\pi_V(y+\sigma_i),i):y\in\Z^{m+p},\,i=1,\ldots,m+p\}$ be the collection of all {\em segments}.
If $\gamma=\pi_V\circ\tilde{T}$ is a strand with vertices in $\Gamma$, we may view the image of $\gamma$ as a collection of  segments. Let's say that {\em we can get from segment $\sigma$ to segment $\sigma'$ in one step} if there is a strand $\gamma$ with $\sigma$ and $\sigma'$ in the image of $\gamma$, and let's say {\em we can get from $\sigma$ to $\sigma'$ in
$n$ steps} if there are segments $\sigma^0=\sigma,\ldots, \sigma^n=\sigma'$ so that we can get from $\sigma^{i-1}$ to $\sigma^i$ in one step for $i=1,\ldots,n$. 

$\Pb$ induces a map that takes segments to collections of segments: if $\sigma$ is a segment defined by the image of $\pi_V\circ \tilde{T}$ on the interval $[t_1,t_2]$, then $\Pb(\sigma)$ is the collection of segments defined by the image of $\pi_V\circ \tilde{\Psi}_{\beta}(\tilde{T})$ on the interval $[\beta t_1,\beta t_2]$. This map has the property that if $\sigma'\in\Pb(\sigma)$ then $\Pb(\Sigma_{\sigma})\subset \Sigma_{\sigma'}$.

\begin{lemma} \label{can get everywhere} If $\beta>1$ is Pisot then there is $K\in\N$ so that for each $\sigma,\sigma'\in\cup_{k\ge K}\Pb^k(\Sigma)$ there is $n\in\N$ so that we can get from $\sigma$ to $\sigma'$ in $n$ steps. 
\end{lemma}

\begin{proof}
Given a segment $\sigma$, let $\Sigma_{\sigma}$ be the collection of all $\sigma'\in\Sigma$ for which there is $n\in\N$ such that we can get from $\sigma$ to $\sigma'$ in $n$ steps. It is clear that:

\begin{enumerate}
 \item  the $\Sigma_{\sigma}$ partition $\Sigma$,
\item for each $\sigma$ there are $\sigma',\sigma''\in\Sigma_{\sigma}$ so that the terminal vertex of $\sigma'$ is the initial vertex of $\sigma$ and the initial vertex of $\sigma''$ is the terminal vertex of $\sigma$, and 
\item if $\sigma'\in\Pb(\sigma)$, then $\Pb(\Sigma_{\sigma})\subset\Sigma_{\sigma'}$.
\end{enumerate}

Let's take $\sigma=(\pi_V(\sigma_1,1))$ so that $\Pb(\Sigma_{\sigma})\subset \Sigma_{\sigma}$. Note that if $\beta$ is a simple Parry number ($m$=0), there is $K\in\N$ so that if $i\in\{1,\ldots,p\}$ then $\sigma\in\Pb^k(\pi_V((\sigma_i,i)))$ for all $k\ge K$; and if $\beta$ is non-simple, it follows from Lemma \ref{initial stability} that there is $K\in\N$ so that $\Pb^k(\pi_V((\sigma_i,i)))\cap\Pb^k(\pi_V((\sigma_j,j)))\ne\emptyset$ for all $i,j\in\{1,\ldots,m+p\}$ and $k\ge K$.  From this (and items (1)-(3) above) it follows that if $\sigma'$ and $\sigma''$ have a common vertex, then there is $\sigma'''$ so that $\Pb^K(\Sigma_{\sigma'})\cup\Pb^K(\Sigma_{\sigma''})\subset\Sigma_{\sigma'''}$ with $K$ as above. From connectedness of $spt(\Sigma):=\pi_V(\tilde{\mathcal{R}}_{\pb})$ we have $\Pb^k(\Sigma)\subset \Sigma_{\sigma}$ for all $k\ge K$.
\end{proof}

Given the segment $\sigma=(x+\pi_V(\sigma_i),i)$, let's call $x+\pi_V(\sigma_i)$ the {\em support} of $\sigma$ (denoted $spt(\sigma)$) and $i$ its {\em type}. We'll say that the strand $\gamma=\pi_V\circ \tilde{T}$ {\em parameterizes $\sigma$ on $[s,s']$} if $\gamma([s,s'])=spt(\sigma)$ and $T([s,s'])$ is the $i$-th loop of $\mathcal{R}_{\pb}$. Note that if $\gamma$ parameterizes $\sigma$ on $[s,s']$, then $\gamma-t$ parameterizes $\sigma$ on $[s-t,s'-t]$. Suppose now that we can get from $\sigma$ to $\sigma'$ in $n$ steps. There are then segments $\sigma=\sigma^0,\ldots,\sigma^n=\sigma'$ and strands $\gamma^1,\ldots,\gamma^n$ so that (on some intervals) $\gamma^i$ parameterizes both $\sigma^{i-1}$ and $\sigma^i$ for $i=1,\ldots,n$. By replacing $\gamma^i$ by $\gamma^i-t_i$, with the appropiately chosen $t_i$ for $i=2,\ldots,n$, we may arrange that: $\gamma^1$ parameterizes $\sigma^0$ on $[s_0,s_0']$ and $\sigma^1$ on $[s_1,s_1']$; $\gamma^2$ parameterizes $\sigma^1$ on $[s_1,s_1']$ and $\sigma^2$ on $[s_2,s_2']$; ... ; and $\gamma^n$ parameterizes $\sigma^{n-1}$ on $[s_{n-1},s_{n-1}']$ and $\sigma_n$ on $[s_n,s_n']$. Then if $\gamma^i=x_i+\pi_V\circ \tilde{S_i}$, the tilings $S_i$ have the property: $S_i(t)=S_{i+1}(t)$ for $t\in [s_i,s_i']$ and $i=1,\ldots,n-1$.

\begin{lemma} \label{rp by steps} Suppose that $T,T'\in\Ob$ and there is $n\in\N$ so that, for each $R>0$, 
there are $S_i=S_i(R)\in\Ob$, $i=1,\ldots,n$,  and $t_i=t_i(R)\in\R$, $i=1,\ldots,n+1$, so that for all $|t|<R$: $(T-t_1)(t)=(S_1-t_1)(t)$; $(S_i-t_{i+1})(t)=S_{i+1}-t_{i+1})(t)$, $i=1,\ldots, n-1$; and $(S_{n}-t_{n+1})(t)=(T'-t_{n+1})(t)$. Then $T\sim_{srp}T'$.
\end{lemma}

\begin{proof} Let $\epsilon>0$ be given. There is then $R>0$ so that if $S,S'\in\Ob$ are such that 
$S(t)=S'(t)$ for $|t|<R$, then $d(\pi_{max}(S),\pi_{max}(S'))<\epsilon/(n+1)$. Since $\R$ acts by isometries on $\hat{\T}^d_{\beta}$  and $\pi_{max}(S-t)=\pi_{max}(S)-t$, we have, for such $S,S'$, that $d(\pi_{max}(S-t),\pi_{max}(S'-t))<\epsilon/(n+1)$ for all $t\in\R$. For this $R$, let the $S_i(R)$ and $t_i(R)$ be as hypothesized. Then: $d(\pi_{max}(T),\pi_{max}(S_1))<\epsilon/(n+1)$; $d(\pi_{max}(S_i)\pi_{max}(S_{i+1}))<\epsilon/(n+1)$, $i=1,\ldots,n-1$; and $d(\pi_{max}(S_n),\pi_{max}(T'))<\epsilon/(n+1)$. Hence $d(\pi_{max}(T),\pi_{max}(T'))<\epsilon$, so $\pi_{max}(T)=\pi_{max}(T')$ and $T\sim_{srp}T'$.

\end{proof}

\begin{proof} (Of Theorem \ref{rp=gs})
Suppose that $T,T'\in \Ob$  with $T\sim_{rp}T'$. There are then corresponding strands $\gamma,\gamma'$ that are strong regionally proximal. By Lemma \ref{lemrp}, $\Pb^k(\gamma)\sim_{srp} \Pb^k(\gamma')$ for all $k\in\Z$ and by so Lemma \ref{B}, $T\sim_{gs} T'$. 

Now suppose that $T\sim_{gs}T'$. Let $B<\infty$, $\tilde{T}$, $\tilde{T}'$, $\tilde{T}_i$, $\tilde{T}'_i$, and $v_i\in\Gamma$ be as in Lemma \ref{reformulation}. 

Recall that, for $S\in\Ob$, $t_*(S)=\sup\{t\le 0:S(t)=*\}$. Let's set $t_1(S):=\inf\{t>t_0(S):S(t)=*\}$. As in Lemma \ref{reformulation} we let $t_{*,j}=t_*(\Pb^j(T))$ and $t'_{*,j}=t_*(\Pb^j(T'))$; correspondingly, we set $t_{1,j}=t_1(\Pb^j(T))$ and $t'_{1,j}=t_1(\Pb^j(T'))$. For each $(k,i)\in\N^2$ let $\gamma_{(k,i)}$ and $\gamma'_{(k,i)}$ be the strands:
$$\gamma_{(k,i)}:=\pi_V\circ\tilde{\Psi}_{\beta}^k(\tilde{T}_i)$$
and 
$$\gamma'_{(k,i)}:=\pi_V\circ\tilde{\Psi}_{\beta}^k(\tilde{T}'_i).$$
Then $\gamma_{(k,i)}$ parameterizes a segment, call it $\sigma_{(k,i)}$, on $[t_{*,k-i},t_{1,k-i}]$ and $\gamma'_{(k,i)}$ parameterizes a segment $\sigma'_{(k,i)}$ on $[t'_{*,k-i},t'_{1,k-i}]$. We observe:
\begin{enumerate}
\item The pair of segments $(\sigma_{(k,i)},\sigma'_{(k,i)})$ depends, up to translation by an element of $\Gamma$, only on $k-i$.
\item $\sigma_{(k,i)},\sigma'_{(k,i)}\in\Pb^k(\Sigma)$, so if $K$ is as in Lemma \ref{can get everywhere} and $k\ge K$, then there is $n$ so that we can get from $\sigma_{(k,i)}$ to $\sigma'_{(k,i)}$ in $n$ steps.
\end{enumerate}

From (1) it follows that there are $K\le i_1<i_2<i_3\cdots$ so that the pairs $(\sigma_{(K,i_j)},\sigma'_{(K,i_j)})$ are all the same, up to translation by elements of $\Gamma$, and then, by (2), there is $n$, independent of $j$, so that we can get from $\sigma_{(K,i_j)}$ to $\sigma'_{(K,i_j)}$ in $n$ steps.

Fix $j\ge1$ and let $\sigma^0=\sigma_{(K,i_j)},\sigma^1,\ldots,\sigma_n=\sigma'_{(K,i_j)}$ be such that we can get from $\sigma^{i-1}$ to $\sigma^i$ in one step for $i=1,\ldots,n$. There are then tilings $S_i$, $i=1,\ldots,n$,  and intervals $[s_i,s_i']$, $i=0,\ldots,n$, with $[s_0,s_0']=[t_{*,K-i_j},t_{1,K-i_j}]$, and lifts $\tilde{S}_i$ so that the strand $\pi_V\circ \tilde{S}_i$ parameterizes $\sigma^{i-1}$ on $[s_{i-1},s'_{i-1}]$ and $\sigma^i$ on $[s_i,s'_i]$, $i=1,\ldots,n$.
\\
\\
{\bf Claim:} $[s_n,s'_n]=[t'_{*,K-i_j},t'_{1,K-i_j}]$.
\\
To see this, note that for any strand $\gamma$, we have: $\pi^u(\gamma(t))=\pi^u(\gamma(0))+t\omega$ for all $t\in\R$. Thus, if $\gamma$ and $\gamma'$ are strands for which $\pi^u(\gamma(t))=\pi^u(\gamma'(t))$ for some $t$ then $\pi^u(\gamma(t))=\pi^u(\gamma'(t))$ for all $t$. Hence, $\pi^u(\pi_V\circ\tilde{\Psi}_{\beta}^K(\tilde{T}_ {i_j})(t))=\pi^u(\pi_V\circ\tilde{S}_i(t))$ for all $t\in\R$ and $i=1,\ldots,n$.
We also know that $\pi_V\circ\tilde{\Psi}_{\beta}^K(\tilde{T}_{i_j})(0)-\pi_V\circ\tilde{\Psi}_{\beta}^K(\tilde{T}'_{i_j})(0)\in E^s$ (if $\gamma$ and $\gamma'$ are such that $\gamma(0)-\gamma'(0)\notin E^s$, then $|\Pb^n(\gamma)(0)-\Pb^n(\gamma')(0)|$ grows without bound as $n\to\infty$, and then so also does $|a_n(\gamma)-a_n(\gamma')|$). Thus $\pi^u(\pi_V\circ\tilde{\Psi}_{\beta}^K(\tilde{T}_ {i_j})(t))=\pi^u(\pi_V\circ\tilde{\Psi}_{\beta}^K(\tilde{T}'_ {i_j})(t))$ for all $t$.
Now $\pi_V\circ\tilde{\Psi}_{\beta}^K(\tilde{T}'_{i_j})$ and $\pi_V\circ\tilde{S}_n$ parameterize the same segment $\sigma'_{(K,i_j)}=\sigma_n$ on the intervals 
$[t'_{*,K-i_j},t'_{1,K-i_j}]$ and $[s_n,s'_n]$, resp., while $\pi^u(\pi_V\circ\tilde{\Psi}_{\beta}^K(\tilde{T}'_ {i_j})(t))=\pi^u(\pi_V\circ\tilde{S}_n)(t)$ for all $t$, so these intervals must be the same, as claimed.
\\
\\
If $r=min\{length(\tau_i):i=1,\ldots,m+p\}$, then each of the intervals $[s_i,s'_i]$ occurring above has length at least $r$. Given $R>0$, let $j$ be large enough so that $\beta^{i_j-K} r>2R$. The hypotheses of Lemma \ref{rp by steps} are then satisfied, with $\Pb^{i_j-K}(S_i)$ playing the role of $S_i$ in the lemma, and $t_i(R):=(\frac{s_{i-1}+s'_{i-1}}{2})\beta^{i_j-K}$.
\end{proof}

\begin{corollary} \label{pihat is pimax} $\hat{\pi}=\pi_{max}:\Ob \to \hat{\T}^d\simeq\hat{\T}^d_{\beta}$.
\end{corollary}

To connect the foregoing with arithetical coding, let $\Pi:X_{\beta}^+\to I$ by $\Pi((x_i)):=\sum_{i=1}^{\infty}x_i\beta^{-i}$. Then $\Pi\circ \sigma=T_{\beta}\circ \Pi$, $\Pi$ is a.e. one-to-one, and so is $\hat{\Pi}:x_{\beta}\to \inv T_{\beta}$. (The unique measure $\eta$ of maximal entropy for $T_{\beta}$ is equivalent to Lebesgue measure (\cite{Hof1, Hof2}), for $\inv T_{\beta}$ we use the measure induced by $\eta$.) The map from $\Ob$ to $[0,1)$ given by $T\mapsto-t_*(T)$ is surjective and  from $-t_*(\Pb(T))=T_{\beta}(-t_*(T))$, there is an induced map $\widehat{-t_*}:\Ob\to\inv T_{\beta}$. The image of $\widehat{-t_*}$ is $\inv T_{\beta}|_{[0,1)}$, which is a full measure subset of $\inv T_{\beta}$. For $T\in\Ob$ and $i\in\Z$, let $x_i(T):=\lfloor-\beta t_*(\Pb^i(T))\rfloor$. Then  $(x_i(T))\in X_{\beta}$ and $\widehat{-t_*}(T)=\hat{\Pi}((x_i(T))$. It follows that $\{x_i(T)):T\in\Ob\}$ has full measure in $X_{\beta}$.

Recall that $\omega$ is the positive right eigenvector of $L$, normalized with $\langle l^t,\omega\rangle=1$, $l=(l_1,\ldots,l_{m+p})$ being the left eigenvector of the substitution matrix $A$ with $l_i=length(\tau_i)$ (note that $\sum_{i=1}^{m+p}l_i=1$). Then $\pi^u(\pi_V(e_i))=l_i\omega$ for each $i\in\{1,\ldots,m+p\}$ and since $\omega$ is `totally irrational' in $V$, $\pi^u:\Gamma\to E^u$ is one-to-one and we have $\Gamma^u:=\pi^u(\Gamma)=\langle l_1\omega,\ldots,l_{m+p}\omega\rangle_{\Z}=\langle l_1,\ldots,l_{m+p}\rangle_{\Z}\omega$. We claim that $\langle l_1,\ldots,l_{m+p}\rangle_{\Z}=\Z[\beta]$. To see this, recall the notation $z^j= T_{\beta}^{j-1}(1)$. Clearly, $z^j\in\Z[\beta]$. Since the endpoints of the $\tau_i$ are contained in $\{0\}\cup\{z^j:j=1,\ldots,m+p\}$, we have $\langle l_1,\ldots,l_{m+p}\rangle_{\Z}\subset \Z[\beta]$. On the other hand, for each $j=1,\ldots,d\le m+p$, there is $i\in \{1,\ldots,m+p\}$ with $max(\tau_i)=z^j$. We have $z^1=1$ and , for $j\ge 2$, $z^j=\beta^{j-1}+q_{j-2}(\beta)$, where $q_{j-2}(x)\in\Z[x]$ has degree $j-2$. Then $\beta^{j}=z^{j+1}-q_{j-1}(\beta)=l_1+l_2+\cdots+l_{j+1}-q_{j-1}(\beta)\in \langle l_1,\ldots,l_{m+p}\rangle_{\Z}$  for all $j$, inductively. We thus have $\Gamma^u=\Z[\beta]\omega$.

Now let $e:=-\pi_V(\sum_{i=1}^{m+p}e_i)\in\Gamma$. From $\pi^u\circ L=\beta\pi^u$ and the above claim, we see that $e,Le,\ldots,L^{d-1}e$ generates $\Gamma$, that is, $\bar{y}:=y+\Gamma$, $y:=e^u$, is a fundamental homoclinic point for $F_L:V/\Gamma\to V/\Gamma$.

For each $i\in\Z$ and $T\in\Ob$ let $p_i(T):=p(\Pb^i(T))=\gamma(t_*(\Pb^i(T)))-\gamma(\beta t_*(\Pb^{i-1}(T))$, $\gamma$ any strand corresponding to $\Pb^i(T)$, be the abelian prefix of $\Pb^i(T)$ and let $a_i:=a_i(\gamma_T)$. Then $(a_i)$ is bounded and $a_i=p_i+La_{i-1}$ for all $i\in\Z$. We have $a_0=p_0+La_{-1}=p_0+Lp_{-1}+L^2a_{-2}=\cdots$. Since the $a_i$ are bounded, $$a_0^s=\sum_{i=0}^{\infty}L^ip^s_{-i},$$ Similarly, $$a^u_0=-\sum_{i=1}^{\infty}L^{-i}p^u_i.$$
 Since $p_i(T)\in\Gamma$, $L^ip^s_{-i}\equiv -L^ip^u_{-i} \mod( \Gamma)$ for $i\ge 0$, and we have $$a_0=a_0(T)\equiv \sum_{i=-\infty}^{\infty}L^i(-p^u_{-i}(T)),\,\mod( \Gamma).$$
 Thus, since $p^u_i(T)=x_i(T)\omega$, $$\pi(T)=a_0(\gamma_T)+\Gamma=-\sum_{i=-\infty}^{\infty}(x_i(T)\beta^{-i}\omega+\Gamma).$$
 
 On the other hand, with fundamental homoclinic point $\bar{y}=e^u=-\omega+\Gamma$ (and $\hat{\bar{y}}=(-\omega+\Gamma,-\beta^{-1}\omega+\Gamma,\ldots)$), we have $h_{\bar{y}}((x_i(T))=\sum_{i=-\infty}^{\infty}x_i(T)\hat{F}_L^{-i}(-\omega+\Gamma, -1/\beta\omega+\Gamma, \ldots)=-(\sum_{-\infty}^{\infty}(x_i(T)\beta^{-i}\omega+\Gamma),\sum_{i=-\infty}^{\infty}(\beta^{-i-1}\omega+\Gamma),\ldots)$. Thus we see that $$\hat{\pi}(T)=h_{\bar{y}}((x_i(T)))$$
 for all $T\in\Ob$.

From Theorem \ref{main theorem} and Corollary \ref{pihat is pimax}, $h_{\bar{y}}$, restricted to $\{(x_i(T)):T\in\Ob\}$ is a.e one-to-one. Since $\{(x_i(T)):T\in\Ob\}$ has full measure in $X_{\beta}$, we have proved Corollary \ref{coding}.

We finish with an argument that a stronger formulation than Property (W) is equivalent to pure discrete spectrum of $(\Ob,\R)$. For $T\in\Ob$ we write $W^s(T)$ for $\{T'\in\Ob:d(\Pb^n(T'),\Pb(T))\to 0 \text{ as }n\to\infty\}$ and for $\hat{\bar{z}}\in\hat{\T}^d_{\beta}$, $W^s(\hat{\bar{z}}):=\{\hat{\bar{z}}':d(\hat{F}^n_M(\hat{\bar{z}}'),\hat{F}^n_M(\hat{\bar{z}}))\to0 \text{ as }n\to\infty\}$. (Here $\hat{F}_M:\hat{\T}^d_{\beta}\to\hat{\T}^d_{\beta}$ is the hyperbolic automorphism induced on the maximal equicontinuous factor of $\Ob$.)
\begin{lemma}  \label{dense} If $(\Ob,\R)$ has pure discrete spectrum and $T,T'\in\Ob$ are such that $T\sim_s T'$, then $\{t:T-t\sim_s T'-t\}$ is open and dense in $\R$
\end{lemma}
\begin{proof} We have previously observed that $T-t\sim_sT'-t$ is an open property. If $T\sim_s T'$ then $d(\pi_{max}(\Pb^k(T-t)),\pi_{max}(\Pb^k(T'-t)))=d(\hat{F}_{M}^k((\pi_{max}(T))-\beta^kt,\hat{F}^k_{M}(\pi_{max}(T'))-\beta^kt)\to 0$ as $k\to\infty$ uniformly in $t$. Suppose there is an interval $J=(t_0-\epsilon, t_0+\epsilon)$ with $T-t\nsim_s T'-t$ for all $t\in J$. Take $k_i\to\infty$ so that $\Pb^{k_i}(T-t_0)\to S\in\Ob$ and $\Pb^{k_i}(T'-t_0)\to S'\in\Ob$. Then $S\sim_{srp} S'$ (since $d(\pi_{max}(S),\pi_{max}(S'))=0$) and $S\cap S'=\emptyset$ (if $B_0[S-t]=B_0[S'-t]$ for some $t$ then $B_0[\Pb^{k_i}(T-(t_0-t/\beta^{k_i}))]=B_0[\Pb^{k_i}(T'-(t_0-t/\beta^{k_i}))]$ for all large $i$, contradicting $T-t'\nsim_s T'-t'$ with $t'=t_0-t/\beta^{k_i}\in J$). But then the coincidence rank of $\pb$ is at least two, so $(\Ob,\R)$ does not have pure discrete spectrum.
\end{proof}
 
A few observations:

\begin{enumerate}
\item For any $T,T'\in\Ob$, $\{t:T'-t\in W^s(T)\}$ is dense in $\R$. (This is a consequence of the primitivity of $\pb$.)
\item $d(t\beta^k\omega,\Gamma)\to 0$ as $k\to\infty$ if and only if $t\in\Z[1/\beta]$. Hence, if $\hat{\bar{0}}=(\bar{0},\bar{0},\ldots)$,  $\hat{\bar{0}}-t\in W^s( \hat{\bar{0}})$ if and only if $t\in\Z[1/\beta]$.
\item $-t_*(\Pb(T))=T_{\beta}(-t_*(T))$ for all $T\in\Ob$. 
\item Since $\pi_{max}^{-1}(\hat{\bar{0}})$ is finite (\cite{BBK}, or \cite{BK}) and $\Pb$-invariant there is $N>0$ so that $\Pb^N(T)=T$ for all $T\in\pi_{max}^{-1}(\hat{\bar{0}})$.
\item Given $T,T'\in\Ob$ with $T'\in W^s(T)$, there is $\epsilon>0$ so that  $T'-t\in W^s(T-t)$ for all $|t|<\epsilon$. (This is a consequence of `local product structure' - see \cite{AP}. Or, invoke the openness of the relation $\sim_s$, which is proved in \cite{BO}.)
\end{enumerate}

\begin{prop} \label{Fin}If $\beta$ is Pisot, then $(\Ob,\R)$ has pure discrete spectrum if and only if for all $t\in\Z[1/\beta]\cap\R^+$ the set $\{t'\in Fin(\beta):t+t'\in Fin(\beta)\}$ is dense in $\R^+$.
\end{prop}
\begin{proof}
Suppose that $(\Ob,\R)$ has pure discrete spectrum and pick $t\in\Z[1/\beta]\cap\R^+$. For given $i\in\{1,\ldots,p\}$, there is $T\in\pi_{max}^{-1}(\hat{\bar{0}})$ so that $T_0^i-t\in W^s(T)$ (this follows from item (2)). Now, $\{t'>0:T_0^i-t'\in W^s(T_0^i)\}$ is dense in $\R^+$ (from (1)) and $\{t'\in\R:T_0^i-t'\sim_s T-t'\}$ is open and dense in $\R$ (Theorem \ref{pds means stab equiv}). Then $C:=\{t'>0:T-t',T_0^i-t'\in W^s(T_0^i)\}$ is dense it $\R^+$ and all these $t'$ are also in $Fin(\beta)$. Indeed, if $t\ge0$ and $T_0^i-t'\in W^s(T_0^i)$, let $k$ be large enough so that $t'/\beta^{kp}<1$. Then $T_{\beta}^{kp+n}(t'/\beta^{kp})=T_{\beta}^{kp+n}(-t_*(T_0^i-t'/\beta^{kp}))=-t_*(\Pb^n(T_0^i-t'))=0$ for $n$ large enough that $\tau_1\in \Pb^n(T_0^i-t')$, so $t'/\beta^{kp}$, and hence $t'$, is in $Fin(\beta)$. We have $T_0^i-t\sim_s T$, so by Lemma \ref{dense}, the set $D:=\{t':T_0^i-t-t'\sim_s T-t'\}$ is open and dense in $\R$. Now $C\cap D$ is dense in $\R^+$ and for $t'\in C\cap D$ we have $T_0^i-t-t'\sim_sT-t'\sim_sT_0^i$, so $t+t'\in Fin(\beta)$ (as above). Thus if $t'\in C\cap D$ then $t'$ and $t+t'$ are both in $Fin(\beta)$.

Now suppose that $F(t):=\{t'\in Fin(\beta):t+t'\in Fin(\beta)\}$ is dense in $\R^+$ for all $t\in\Z[1/\beta]\cap\R^+$ and let $T\in \pi_{max}^{-1}(\hat{\bar{0}})$. Fix $i\in\{1,\ldots,p\}$ and let $t\in\R^+$ be so that $T_0^i-t\sim_s T$. Then $t\in\Z[1/\beta]$. Let $\epsilon>0$ be small enough so that if $|t'|<\epsilon$, then $T_0^i-t-t'\sim_s T-t'$. Note that if $t'\in Fin(\beta)$, then $t'/\beta^{kp}\in Fin(\beta)$ and we have  $0=T_{\beta}^{n+kp}(t'/\beta^{kp})=T_{\beta}^{n+kp}(-t^*(T_0^i-t'/\beta^{kp} ))=-t^*(\Pb^n(T_0^i-t'))$ for large $n$ and $k$ big enough so that $t'/\beta^{kp}\in [0,1)$. Thus, $T_0^i-t'\in W^s(T_0^j)$ for some $j\in\{1,\ldots,p\}$. Then for $t'\in F(t)\cap [0,\epsilon)$ we have $T_0^i-t'\sim_s T_0^j$, $T_0^i-(t+t')\sim_s T-t'$ and $T_0^i-(t+t')\sim_s T_0^k$ for some $j,k\in\{1,\ldots,p\}$. Since $T_0^j-t''\sim_sT_0^k-t''$ for all $t''>0$, we see that $T_0^i\sim_s T$ densely on $[0,\epsilon)$. Using the $\Pb$-periodicity of $T_0^i$ and $T$, this means that $W(T):=\{t'>0:T_0^i-t'\sim T-t'\}$ is open and dense in $\R^+$. Thus, for each $T,T'\in \pi_{max}^{-1}(\hat{\bar{0}})$, $W(T)\cap W(T')\ne 0$, and then $T\cap T'\ne \emptyset$. Thus the coincidence rank of $\pb$ is 1 and $(\Ob,\R)$ has pure discrete spectrum.

\end{proof}

It follows from Theorem \ref{main theorem} and Proposition \ref{Fin} that all Pisot $\beta$ are weakly finitary, proving Corollary \ref{weak}.

\end{document}